\UseRawInputEncoding
\documentclass[11pt, reqno, oneside]{amsart}
\usepackage[T1]{fontenc}
\usepackage[utf8]{inputenc}
\usepackage[british]{babel}
\usepackage{amsmath,amssymb,amsthm,mathtools}
\usepackage{stmaryrd}
\usepackage{geometry}
\usepackage{xcolor}
\usepackage{graphicx}
\usepackage{hyperref}
\usepackage{enumitem}
\hypersetup{colorlinks=true,linkcolor=blue,citecolor=blue,urlcolor=blue}
\usepackage{caption, subcaption}
\usepackage{comment}
\usepackage{cleveref}
\usepackage{tikz-cd}
\usepackage{adjustbox}
\usepackage{thmtools}

\tikzset{
curarrow/.style={
rounded corners=10pt,
execute at begin to={every node/.style={fill=red}},
to path={-- ([xshift=50pt]\tikztostart.center)
  |- (#1) node[fill=white] {$\scriptstyle \delta$}
  -| ([xshift=-40pt]\tikztotarget.center)
  -- (\tikztotarget)}
  }
}

\tikzset{
curvararrow/.style={
rounded corners=10pt,
execute at begin to={every node/.style={fill=red}},
to path={-- ([xshift=70pt]\tikztostart.center)
  |- (#1) node[fill=white] {$\scriptstyle \delta$}
  -| ([xshift=-40pt]\tikztotarget.center)
  -- (\tikztotarget)}
  }
}

\DeclarePairedDelimiterX\Set[1]\{\}{%

#1
}

\title{Hamiltonian Group Actions in Cosymplectic Geometry}
\author{Eva Miranda}
\address{Eva Miranda,
Laboratory of Geometry and Dynamical Systems \& SYMCREA, Department of Mathematics, EPSEB, Universitat Polit\`{e}cnica de Catalunya-IMTech
in Barcelona and
\\ CRM Centre de Recerca Matem\`{a}tica, Campus de Bellaterra
Edifici C, 08193 Bellaterra, Barcelona.}
\thanks{Corresponding author: Eva Miranda. Email: \texttt{eva.miranda@upc.edu}}
\thanks{Eva Miranda is funded by the Catalan Institution for Research and Advanced Studies via an ICREA Academia Prize 2021 and by the Alexander von Humboldt Foundation via a Friedrich Wilhelm Bessel Research Award.}
\thanks{Both authors are supported by the Spanish State
Research Agency, through the Severo Ochoa and Mar\'{\i}a de Maeztu Program for Centers and Units
of Excellence in R\&D (project CEX2020-001084-M), and under grant reference PID2023-146936NB-I00 funded by MICIU/AEI/10.13039/501100011033 and by ERDF/EU}
\author{Pablo Nicolás}
\address{Pablo Nicolás,
Centre de Recerca Matemàtica, CRM \& Laboratory of Geometry and Dynamical Systems and SYMCREA research unit, Department of Mathematics, Universitat Polit\`{e}cnica de Catalunya, Barcelona}
\thanks{Pablo Nicolás is supported under an MDM-FPI contract with reference code PRE2022-102974.}
\date{\today}

\declaretheoremstyle[
  spaceabove=1em, spacebelow=1em,
  headfont=\bfseries,
  notefont=\bfseries, notebraces={(}{)},
  bodyfont=\itshape,
  postheadspace=0.5em
]{thmlike}

\declaretheoremstyle[
  spaceabove=1em, spacebelow=1em,
  headfont=\normalfont,
  notefont=\normalfont, notebraces={(}{)},
  bodyfont=\normalfont,
  postheadspace=1em,
  qed=$\blacksquare$
]{prooflike}

\declaretheoremstyle[
  spaceabove=1em, spacebelow=1em,
  headfont=\bfseries,
  notefont=\bfseries, notebraces={(}{)},
  bodyfont=\normalfont,
  postheadspace=0.5em
]{deflike}

\declaretheorem[style=thmlike, numberwithin=section]{theorem}
\declaretheorem[style=thmlike, sibling=theorem]{corollary}
\declaretheorem[style=deflike, sibling=theorem]{definition}
\declaretheorem[style=thmlike, sibling=theorem]{lemma}
\declaretheorem[style=thmlike, sibling=theorem]{proposition}
\declaretheorem[style=deflike, sibling=theorem]{example}

\declaretheorem[style=deflike, sibling=theorem]{remark}

\declaretheorem[style=deflike, sibling=theorem, numbered=no, name=Definition]{definition*}
\declaretheorem[style=deflike, sibling=theorem, numbered=no, name=Example]{example*}

\newcommand{\g}{\mathfrak g}

\newcommand*\diff{\mathop{}\!\mathrm{d}}

\makeatletter

\newcommand{\c@otimes}[2]{%
  \vbox{
    \ialign{##\cr
      \hidewidth$\m@th#1{}_\frown$\kern-\scriptspace\hidewidth\cr
      \noalign{\nointerlineskip\kern-1pt}
      $\m@th#1\otimes$\cr
    }%
  }%
}
\makeatother

\begin{document}

\begin{abstract}
    We develop a theory of Hamiltonian group actions on cosymplectic manifolds. These odd-dimensional manifolds combine a codimension-one symplectic foliation with a distinguished Reeb direction, and arise naturally both in stable Hamiltonian geometry and as critical hypersurfaces of \(b\)-symplectic manifolds.
    
    Our approach is based on a compact symplectic thickening process: every cosymplectic manifold \((M,\alpha,\beta)\) gives rise to a symplectic manifold
    \[
    \bigl(M\times \mathbb{S}^1,\ \beta + \diff\theta\wedge\alpha\bigr).
    \]
    We prove that Hamiltonian cosymplectic actions lift canonically to Hamiltonian symplectic actions on this symplectic manifold. This provides a systematic bridge between equivariant symplectic geometry and the cosymplectic setting.
    
    Using this bridge, we establish cosymplectic analogues of convexity theorems for torus actions, Delzant theorem for toric actions, ABBV localization, Duistermaat--Heckman formulas, and Kirwan surjectivity. The resulting formulas are not merely formal pullbacks from the symplectic case: the Reeb direction appears explicitly through the factor \(\alpha\), and, in the mapping-torus case, the localization and volume formulas are governed by the modular period together with the equivariant geometry of the symplectic fiber. We also explain how these results apply to Hamiltonian geometry on the critical hypersurfaces of \(b\)-symplectic manifolds.
\end{abstract}

\dedicatory{To the memory of Marisa Fernández\textsuperscript{1} with admiration.}

\maketitle

\footnotetext[1]{From the first author: Marisa Fernández passed away last November.  Marisa has been a guiding light throughout my career and a dear friend, always generous with her wisdom and advice. I have learned immensely from her and from her experience. This article on cosymplectic geometry is written in honour of her persistence, her vision, and her remarkable work. }

\refstepcounter{footnote}

\section{Introduction}
Cosymplectic geometry is the odd-dimensional counterpart of symplectic geometry, in which a manifold carries both a codimension-one symplectic foliation and a distinguished transverse direction. More precisely, a cosymplectic manifold is a $(2n+1)$-dimensional manifold endowed with a pair of closed forms $(\alpha,\beta)$ such that
\(
    \alpha\wedge\beta^n\neq 0.
\)
The $1$-form $\alpha$ defines a codimension-one foliation, and the restriction of $\beta$ to its leaves is symplectic. Thus, cosymplectic manifolds lie naturally at the intersection of symplectic geometry, foliation theory, Poisson geometry, and Hamiltonian dynamics.

A basic feature of a cosymplectic structure is the existence of a canonical Reeb vector field $R$, uniquely determined by
\[
\iota_R\alpha=1,\qquad \iota_R\beta=0.
\]
The Reeb direction is transverse to the characteristic symplectic foliation and plays a structural role throughout the theory. In the compact case, under mild hypotheses, the geometry is governed by a mapping-torus description: compact cosymplectic manifolds with a compact leaf fibre over $\mathbb{S}^1$, with symplectic fibre and symplectic monodromy. This makes cosymplectic manifolds particularly amenable to both symplectic and dynamical methods.

Cosymplectic manifolds also arise naturally in singular symplectic geometry. If $(X^{2n+2},Z,\omega_b)$ is a $b$-symplectic manifold (cf. \cite{guillemin_symplectic_2014}), then near a connected component $Z$ of the critical hypersurface one has the normal form
\[
\omega_b = \frac{\diff r}{r}\wedge\alpha+\beta,
\]
where $(\alpha,\beta)$ is a cosymplectic structure on $Z$. Thus, the critical hypersurface of a $b$-symplectic manifold is canonically cosymplectic. Hamiltonian questions in $b$-symplectic geometry therefore induce natural Hamiltonian questions on cosymplectic manifolds.

The purpose of this paper is to develop a systematic Hamiltonian theory for group actions on cosymplectic manifolds. In symplectic geometry, Hamiltonian actions are governed by moment maps, equivariant cohomology, convexity theorems, localization formulae, and reduction. We show that these tools admit natural cosymplectic analogues. The central mechanism is a compact symplectic thickening: given a cosymplectic manifold $(M,\alpha,\beta)$, the product
\[
\widehat M=M\times \mathbb{S}^1
\]
carries the symplectic form
\[
\widehat\omega = \beta + \diff \theta\wedge\alpha.
\]
Moreover, Hamiltonian cosymplectic actions lift canonically to Hamiltonian symplectic actions on $(\widehat{M},\widehat{\omega})$, with the same moment map pulled back from $M$.

This construction provides a bridge between equivariant symplectic geometry and the cosymplectic setting. It allows us to transfer classical results such as convexity, ABBV localization, and Kirwan surjectivity. At the same time, the resulting formulae retain genuinely cosymplectic features: the Reeb direction appears through the factor $\alpha$, and, in the mapping-torus case, the formulas reduce to symplectic expressions on the fibre multiplied by the modular period.

We begin by formulating Hamiltonian cosymplectic actions in the Cartan model. With the convention
\[
    \iota_{X_H}\omega = - \diff H,
\]
a Hamiltonian $G$-action with moment map $\mu\colon M\to\mathfrak{g}^*$ satisfies
\[
    \iota_{X^\#}\beta = - \diff \mu^X,
    \qquad
    \mu^X=\langle\mu,X\rangle.
\]
We prove that this condition is equivalent to the existence of an equivariantly closed cosymplectic extension
\[
    \underline{\beta}=\beta-\mu.
\]
More generally, we characterize equivariant closedness of leafwise symplectic forms in terms of the existence of leafwise moment maps. This provides the cohomological framework for cosymplectic Hamiltonian actions and cosymplectic reduction.

Our first application is a convexity theorem. If a compact torus acts in a Hamiltonian fashion on a compact cosymplectic manifold, then the image of the moment map is a convex polytope. This follows by passing to the compact symplectic thickening and applying the Atiyah--Guillemin--Sternberg convexity theorem.

We then establish localization formulae for Hamiltonian group actions in the cosymplectic setting. For a compact cosymplectic manifold $(M,\alpha,\beta)$ with a Hamiltonian action of a compact connected Lie group $G$, we obtain
\[
\int_M \alpha\wedge \underline{\tau}\, \mathrm{e}^{\beta-\mu}
=
\sum_{C\in\pi_0(M^G)}
\int_C
\alpha\wedge
\frac{i_C^*\bigl(\underline{\tau}\, \mathrm{e}^{\beta-\mu}\bigr)}{e_G(\nu_C)}.
\]
This is the cosymplectic analogue of the Atiyah--Bott--Berline--Vergne localization formula. The appearance of the factor $\alpha$ reflects the transverse Reeb direction and is the main structural difference from the symplectic formula.

In the mapping-torus case, the localization formula becomes especially explicit. If $M$ fibres over $\mathbb{S}^1$ with compact symplectic fibre $L$ and modular period $\operatorname{mp}(\alpha,\beta)$, then the relevant integrals over $M$ reduce to integrals over $L$ multiplied by $\operatorname{mp}(\alpha,\beta)$. This yields cosymplectic analogues of the Duistermaat--Heckman formula and stationary phase formulas, expressed in terms of the equivariant geometry of the symplectic fibre and the monodromy data.

Finally, we prove a Kirwan-type surjectivity theorem for Hamiltonian cosymplectic reduction. Under the usual regularity and smoothness assumptions on the reduced space, the Kirwan map
\[
\mathrm{H}_G^\bullet(M)\longrightarrow \mathrm{H}^\bullet(M\sslash G)
\]
is surjective. The proof again proceeds through the compact symplectic thickening, but the resulting statement applies intrinsically to the cosymplectic reduced space.

We also study the variation of the cosymplectic cohomology classes under Hamiltonian reduction. In the symplectic case, the Duistermaat--Heckman theorem describes the linear variation of the cohomology class of the reduced symplectic form as the value of the moment map varies through regular values. In the cosymplectic setting, there are two cohomology classes to track, namely the Reeb class $[\alpha]$ and the leafwise symplectic class $[\beta]$. Using the compact symplectic thickening, we prove that the reduced Reeb class remains constant, while the reduced leafwise symplectic class varies linearly with the moment value, with slope given by the Chern class of the corresponding principal torus bundle. Thus the Duistermaat--Heckman variation formula admits a natural cosymplectic counterpart:\[[\alpha_\xi]=[\alpha_{\xi_0}],\qquad
[\beta_\xi]=[\beta_{\xi_0}]+\langle \xi-\xi_0,\Omega\rangle .
\]
This shows, in particular, that the symplectic geometry of the leaves changes according to the usual Duistermaat--Heckman affine law.

These results show that cosymplectic manifolds support a Hamiltonian theory closely parallel to the symplectic one, while retaining distinctive odd-dimensional features. The compact symplectic thickening supplies the bridge to classical equivariant symplectic geometry; the Reeb direction, characteristic foliation, modular period, and mapping-torus structure provide the genuinely cosymplectic content of the theory.

\section{Preliminaries} \label{sec:preliminaries}

\subsection{Cosymplectic geometry} \label{ssec:cosymplectic-geometry}

A cosymplectic manifold is a triple $(M^{2n + 1}, \alpha, \beta)$ with $\alpha \in \Omega^1(M)$ and $\beta \in \Omega^2(M)$ closed forms satisfying the condition
\begin{equation} \label{eq:cosymplectic-volume}
    \alpha \wedge \beta^n \neq 0
\end{equation}
As a consequence, cosymplectic manifolds are orientable and are equipped with a canonical choice of volume. We may refer as well to the pair $(\alpha, \beta)$ and call it a cosymplectic structure on $M$.

From the fact that $\alpha$ is closed and nowhere vanishing, the distribution $\ker \alpha \subset \mathrm{T}M$ has constant rank and is integrable. Hence, we obtain a foliation $\mathcal{F}$ called the \emph{characteristic foliation} of $M$. Furthermore, this foliation turns out to be symplectic in the following sense. From the natural inclusion $i \colon \mathrm{T} \mathcal{F} \to \mathrm{T}M$, the dual map induces a surjection $i^* \colon \mathrm{T}^* M \to \mathrm{T}^* \mathcal{F}$ which can be extended to a map $j \colon \Omega^\bullet(M) \to \Omega^\bullet(\mathcal{F})$. Similarly to the smooth setting, $\Omega^k(\mathcal{F})$ denotes the space of sections of $\bigwedge^k \mathrm{T}^* \mathcal{F}$. The map sits inside the short exact sequence
\begin{equation} \label{eq:gmp-short-sequence}
    \begin{tikzcd}[column sep=small]
    	0 & {\alpha \wedge \Omega^{\bullet - 1}(\mathcal{F})} & {\Omega^{\bullet}(M)} & {\Omega^{\bullet}(\mathcal{F})} & 0
    	\arrow[from=1-1, to=1-2]
    	\arrow[from=1-2, to=1-3]
    	\arrow["j", from=1-3, to=1-4]
    	\arrow[from=1-4, to=1-5]
    \end{tikzcd}
\end{equation}
and, since $\alpha$ is closed, the operator $\diff \colon \Omega^\bullet(M) \to \Omega^{\bullet + 1}(M)$ descends to a well-defined operator $\diff_{\mathcal{F}} \colon \Omega^\bullet(\mathcal{F}) \to \Omega^{\bullet + 1}(\mathcal{F})$ called the \emph{foliated differential}. In this setting, the form $j \beta \in \Omega^2(\mathcal{F})$ is symplectic, that is, $\diff_{\mathcal{F}}$-closed and non-degenerate.

Cosymplectic manifolds are inherently dynamical objects: there exists a unique vector field $R \in \mathfrak{X}(M)$, called the \emph{Reeb vector field}, satisfying the set of equations
\begin{equation}
    \iota_R \alpha = 1, \quad \iota_R \beta = 0.
\end{equation}
The Reeb vector field is transverse to the symplectic foliation $\mathcal{F}$ and maps leaves to leaves. These properties were used by Guillemin, Miranda, and Pires, building on previous characterizations of Tischler~\cite[Thm.\,1]{tischler_fibering_1970}, to prove a structure theorem for compact cosymplectic manifolds under mild topological assumptions (see Li~\cite[Thm.\,1]{li_topology_2008} for a similar result for co-Kähler manifolds).

\begin{theorem}[{Guillemin--Miranda--Pires \cite[Thm.\,13]{guillemin_codimension_2011}}] \label{thm:mapping-torus}
    Let $(M,\eta,\omega)$ be a compact cosymplectic manifold with a compact leaf $L \subset M$. Then, all leaves are compact and $M$ is diffeomorphic to the \emph{mapping torus} of a symplectomorphism
    \begin{equation*}
        M \cong \frac{L\times[0,c]}{(x,0)\sim(\varphi(x),c)} \eqqcolon M_\varphi,
    \end{equation*}
    where $\varphi \coloneqq \Phi_R^c \colon (L,\omega_L) \to (L,\omega_L)$ is a symplectomorphism. In particular, $M$ fibres over $\mathbb{S}^1$ with fibre $L$ and monodromy $\varphi$.
\end{theorem}

Under these assumptions, the minimum positive number $c \in \mathbb{R}^+$ such that $\Phi_R^c \colon M \to M$ maps all leaves to themselves is an invariant of the cosymplectic structure $(\alpha, \beta)$, called the \emph{modular period}. We will denote it by $\operatorname{mp}(\alpha, \beta)$.

\subsection{Equivariant cohomology}
\label{ssec:equiv-cohomology}

Equivariant cohomology, introduced by Borel \cite{borel_1960_seminar}, provides a robust way to study the topology of quotient spaces $M/G$, even when the group action is not free. The key idea is to replace $M$ by its so-called \emph{homotopy quotient}
\[
M_G \coloneqq (M \times \mathrm{E}G)/G,
\]
where $\mathrm{E}G$ is a weakly contractible space with a free $G$-action (whose existence is due to Milnor~\cite[Thm.\,3.1]{milnor_construction_1956}). The space $M_G$ comes equipped with a canonical map $M_G \to M/G$ and retains information about stabilisers that is lost in the naive quotient. The \emph{equivariant cohomology} of $M$ is then defined as
\begin{equation}\label{eq:def-equiv-cohomology-homotopy}
    \mathrm{H}_G^\bullet(M) \coloneqq \mathrm{H}^\bullet(M_G).
\end{equation}

For compact, connected Lie groups, equivariant cohomology admits a differential model, known as the \emph{Cartan model}, which computes its torsion-free part (cf.~\cite[Thm.\,A.1]{tu_introductory_2020}). The Cartan complex is
\begin{equation}\label{eq:cartan-complex-def}
    \Omega_G^\bullet(M) \coloneqq \big(\Omega^\bullet(M)\otimes \operatorname{Sym}^\bullet(\mathfrak{g}^*)\big)^G,
\end{equation}
graded by $\deg(\omega \otimes p)=i+2j$ for $\omega\in\Omega^i(M)$ and $p\in\operatorname{Sym}^j(\mathfrak{g}^*)$. The $G$-action is given by pullback on forms and the coadjoint action on $\operatorname{Sym}^\bullet(\mathfrak{g}^*)$.

Fixing a basis $X_1,\dots,X_m$ of $\mathfrak{g}$ with dual basis $\sigma^1,\dots,\sigma^m$, and writing $X^\#$ for the fundamental vector field, the equivariant differential is
\begin{equation}\label{eq:equiv-diff-tensor-def}
    \diff_G = \diff \otimes \operatorname{id} - \sum_{i=1}^m \iota_{X_i^\#}\otimes \sigma^i.
\end{equation}

More generally, following \cite[Sec.\,2.1]{ginzburg_equivariant_1999}, let $(A^\bullet,\diff)$ be a cochain complex of Fréchet $G$-modules. It is called a \emph{$G$-differential complex} if there exist operators $\iota_X$ satisfying
\begin{subequations}
\begin{align}
    \iota_X \iota_Y + \iota_Y \iota_X &= 0, \\
    g \iota_X g^{-1} &= \iota_{\operatorname{Ad}_g X}, \\
    \mathcal{L}_X &= \diff \iota_X + \iota_X \diff.
\end{align}
\end{subequations}
In this setting, one defines the Cartan complex and differential by analogy with \eqref{eq:cartan-complex-def}–\eqref{eq:equiv-diff-tensor-def}, and the resulting cohomology is denoted $\mathrm{H}_G^\bullet(A)$.

\begin{example}
If $G=\{e\}$ acts trivially, then $\mathfrak{g}=0$ and $\operatorname{Sym}^\bullet(\mathfrak{g}^*)=\mathbb{R}$. The Cartan complex reduces to $A^\bullet$ and $\diff_G=\diff$, so equivariant cohomology coincides with ordinary cohomology.
\end{example}

\begin{example}\label{ex:equiv-symp-form}
Let $M$ be a $G$-manifold. An \emph{equivariant symplectic form} is a non-degenerate, $\diff_G$-closed element $\underline{\omega}\in\Omega_G^2(M)$, which can be written as
\[
\underline{\omega} = \omega - \mu = \omega - \sum_{i=1}^m f_i \otimes \sigma^i.
\]
Non-degeneracy implies that $\omega$ is symplectic. A direct computation yields
\[
    \diff_G \underline{\omega}
    = \diff \omega - \sum_{i=1}^m (\iota_{X_i^\#}\omega + \diff f_i)\otimes \sigma^i,
\]
so $\diff_G\underline{\omega}=0$ is equivalent to $\diff\omega=0, \iota_{X_i^\#}\omega = -\diff f_i$. Thus, $\underline{\omega}$ encodes a symplectic form together with a weakly Hamiltonian $G$-action, with moment map $\mu^X=\sum_{i = 1}^m f_i\,\sigma^i(X)$.

Finally, $G$-invariance of $\underline{\omega}$ is equivalent to equivariance of the moment map:
\[
\rho_g^*\mu^X = \mu^{\operatorname{Ad}_g X},
\]
so the action is Hamiltonian. Thus, we observe that there exists a correspondence between equivariantly closed extensions $\underline{\omega}$ of a symplectic form $\omega$ and moment maps for the $G$-action, a result obtained by Atiyah and Bott~\cite[Prop.\,6.18]{atiyah-momentmap-1984}
\end{example}

\section{Equivariant closedness, moment maps, and cosymplectic reduction}

We have seen in \Cref{ex:equiv-symp-form} how the concept of equivariant symplectic forms corresponds to Hamiltonian $G$-spaces. Drawing inspiration from this relationship, we define the notion of equivariant cosymplectic structure.

\begin{definition} \label{def:equivariant-cosymplectic}
    Let $M$ be a smooth manifold with $\dim M = 2n + 1$ endowed with an action of a compact, connected Lie group $G$. An \emph{equivariant cosymplectic structure} is a pair $(\underline{\alpha}, \underline{\beta})$, with $\underline{\alpha} \coloneqq \alpha \in \Omega_G^1(M)$ and $\underline{\beta} \coloneqq \beta - \mu \in \Omega_G^2(M)$ equivariantly closed forms satisfying the condition $\alpha \wedge \beta^n \neq 0$.
\end{definition}

Equivariant cosymplectic structures have already been introduced in \cite[Def.\,4.1]{miranda_equivariant_2026}, where their basic properties are discussed. We recall that, if $(\underline{\alpha}, \underline{\beta})$ is an equivariant cosymplectic structure, then the underlying triple $(M, \alpha, \beta)$ is a cosymplectic manifold \cite[Prop.\,4.2]{miranda_equivariant_2026}, the $G$-action is tangent to the characteristic foliation $\mathcal{F}$ \cite[p.\,10]{miranda_equivariant_2026}, and the sequence \eqref{eq:gmp-short-sequence} admits an equivariant extension to the short exact sequence
\begin{equation} \label{eq:foliated-ses-cosymplectic}
    \begin{tikzcd}
    	0 & {\alpha \wedge \Omega_{G}^{\bullet - 1}(\mathcal{F})} & {\Omega_G^\bullet(M)} & {\Omega_G^\bullet(\mathcal{F})} & 0
    	\arrow[from=1-1, to=1-2]
    	\arrow[from=1-2, to=1-3]
    	\arrow[from=1-3, to=1-4, "{j \otimes \operatorname{id}_{\operatorname{Sym}^\bullet(\mathfrak{g}^*)}}" {yshift=5pt}]
    	\arrow[from=1-4, to=1-5]
    \end{tikzcd}
\end{equation}
of Cartan complexes \cite[Eq.\,26]{miranda_equivariant_2026}.

In the symplectic setting, the condition of being equivariantly closed is equivalent to the action being Hamiltonian (cf.~\Cref{ex:equiv-symp-form}). We now give the corresponding characterization in the foliated setting, which shows that equivariant leafwise closedness encodes the existence of a (leafwise) moment map.

\begin{theorem}\label{thm:leafwise-equiv-closed-characterisation}
    Let $\beta \in \Omega^2(M)$ be a $G$-invariant leafwise symplectic form, i.e.\ $j\beta\in\Omega^2(\mathcal F)$ is closed and non-degenerate, and fix $R$ with $\iota_R\alpha=1$ and $\iota_R\beta=0$. The following are equivalent:
    \begin{enumerate}[label=(\roman*)]
        \item\label{it:leafwise-equiv-closed-1} There exists $\underline{\beta}\in\Omega_G^2(M)$ representing $\beta$ such that $j(\diff_G\underline{\beta})=0$.
        \item\label{it:leafwise-equiv-closed-2} For every $X\in\g$ there exists $\mu^X \in \mathcal{C}^\infty(M)$ such that
        \[
        \iota_{X^\#}j\beta=-j\diff \mu^X.
        \]
        \item\label{it:leafwise-equiv-closed-3} There exists $\mu\colon M\to\g^*$ such that
        \[
        \iota_{X^\#}\beta=-\diff\langle\mu,X\rangle+\phi_X\alpha
        \]
        for some $\phi_X\in C^\infty(M)$.
        \item\label{it:leafwise-equiv-closed-4} There exists $\mu\in C^\infty(M)\otimes\g^*$ such that
        \begin{equation}\label{eq:cartan-leaf}
        \diff_G(\beta-\mu)=\alpha\wedge\underline{\gamma}
        \quad\text{for some }\underline{\gamma}\in\Omega_G^2(M).
        \end{equation}
    \end{enumerate}
\end{theorem}

\begin{proof}
    The equivalence between items \ref{it:leafwise-equiv-closed-1} and \ref{it:leafwise-equiv-closed-4} follows from the exactness of the sequence \eqref{eq:foliated-ses-cosymplectic}: the condition $j(\diff_G\underline{\beta})=0$ is equivalent to $\diff_G\underline{\beta}=\alpha\wedge\underline{\gamma}$.
    
    For the equivalence between \ref{it:leafwise-equiv-closed-1} and \ref{it:leafwise-equiv-closed-2}, write $\underline{\beta}=\beta-\mu$. Then $j(\diff_G\underline{\beta}) = -\iota_{X_i^\#}j\beta - \diff_{\mathcal F} f_i$,
    so its vanishing is equivalent to $\iota_{X_i^\#}j\beta = -\diff_{\mathcal F} f_i$. Linearity yields the condition for all $X\in\g$.
    
    For items \ref{it:leafwise-equiv-closed-2} and \ref{it:leafwise-equiv-closed-3}, the identity $\iota_{X^\#}j\beta=-\diff_{\mathcal F} \mu^X$ lifts to $\iota_{X^\#} \beta = -\diff \mu^X + \phi_X\alpha$ by exactness of the foliated sequence \eqref{eq:foliated-ses-cosymplectic}.
    
    Finally, the equivalence between items \ref{it:leafwise-equiv-closed-3} and \ref{it:leafwise-equiv-closed-4} is the Cartan reformulation of leafwise equivariant closedness.
\end{proof}

However, equivariant cosymplectic structures are not simply leafwise equivariantly closed, but also equivariantly closed. This condition will eventually (cf. \Cref{prop:equiv-cosymplectic-hamiltonian-correspondence}) imply the existence of a genuine moment map rather than a leafwise one. In any case these conditions naturally lead to reduction.

We briefly recall cosymplectic reduction following \cite{albert_reduction_1989}. Let $(M,\alpha,\beta)$ be a cosymplectic manifold with a $G$-action preserving $\alpha$ and $\beta$. A map $\mu\colon M\to\g^*$ is a \emph{moment map} if
\begin{subequations}\label{eq:def-hamiltonian-cosymplectic}
\begin{align}
    \iota_{X^\#}\alpha &= 0, \\
    \iota_{X^\#}\beta &= -\diff\mu^X, \\
    \mathcal L_R \mu^X &= 0, \label{eq:cos-mom-map-reeb-inv}
\end{align}
\end{subequations}
for all $X\in\g$.

\begin{theorem}[{\cite[p.\,639]{albert_reduction_1989}}] \label{thm:cos-reduction}
    Let $(M,\alpha,\beta)$ be a connected cosymplectic manifold with a Hamiltonian $G$-action and moment map $\mu$. Let $\xi\in\mathfrak{g}^*$ be a regular value, and $G_\xi$ the stabiliser of $\xi$. If the induced action of $G_\xi$ on $\mu^{-1}(\xi)$ is fibrating, then the reduced space
    \[
    M_\xi \coloneqq \mu^{-1}(\xi)/G_\xi
    \]
    inherits a cosymplectic structure $(\alpha_\xi,\beta_\xi)$.
\end{theorem}

Following our original motivation for the introduction of equivariant cosymplectic structures, we now make explicit the equivalence between Hamiltonian actions and equivariant cosymplectic structures in the spirit of \Cref{ex:equiv-symp-form}.

\begin{proposition} \label{prop:equiv-cosymplectic-hamiltonian-correspondence}
    Let $(M^{2n+1},\alpha,\beta)$ be cosymplectic and $G$ compact and connected. Hamiltonian $G$-actions on $M$ are in one-to-one correspondence with equivariant cosymplectic structures $(\underline{\alpha},\underline{\beta})$ extending $(\alpha,\beta)$ and satisfying $\mathcal{L}_R\underline{\beta}=0$.
\end{proposition}

\begin{proof}
    Given a Hamiltonian action with moment map $\mu$, define
    \[
        \underline{\alpha}=\alpha,
        \qquad
        \underline{\beta}=\beta-\mu.
    \]
    Then $\diff_G\underline{\alpha}=0$ is equivalent to $\diff\alpha=0$ and $\iota_{X^\#}\alpha=0$, while $\diff_G\underline{\beta}=0$ is equivalent to $\diff\beta=0$ and $\iota_{X^\#}\beta=-\diff\mu^X$. The condition $\mathcal L_R\underline{\beta}=0$ is precisely $\mathcal L_R\mu=0$.
    
    Conversely, any such pair $(\underline{\alpha},\underline{\beta})$ recovers a Hamiltonian action with moment map $\mu$.
\end{proof}

Thus equivariant cosymplectic structures provide an algebraic framework for cosymplectic reduction.

\section{A symplectic embedding of cosymplectic manifolds} \label{sec:symplectic-thickening}

In this section we introduce the technique of symplectic thickening of a cosymplectic manifold. This technique produces a compact symplectic manifold whenever the base manifold is compact, and furthermore allows one to extend Hamiltonian actions in the sense of \Cref{thm:cos-reduction} to Hamiltonian actions in the symplectic setting. We will use the tight interplay between cosymplectic manifolds and their compact symplectic thickenings in \Cref{sec:convexity-cosymplectic,sec:localization-cosymplectic,,sec:kirwan-cosymplectic,,sec:var-coh-class}.

\begin{proposition}[compact symplectic thickening of a cosymplectic manifold] \label{prop:symplectic-thickening}
    Let $(M^{2n+1},\alpha,\beta)$ be a cosymplectic manifold, that is,
    \[
        \diff \alpha=0,\qquad \diff \beta=0,\qquad \alpha\wedge \beta^n\neq 0.
    \]
    Let $\theta$ denote the angular coordinate on $\mathbb{S}^1=\mathbb{R}/\mathbb{Z}$, and let $\diff\theta$ be the standard nowhere vanishing closed $1$-form on $\mathbb{S}^1$. Then the product manifold $\widehat{M} \coloneqq M \times \mathbb{S}^1$ carries a natural symplectic form\footnote{
        To be completely precise, the form should be defined as $\omega = \operatorname{pr}_1^* \beta + \operatorname{pr}_2^*(\diff \theta) \wedge \operatorname{pr}_1^* \alpha$. However, we will assume this expression and drop the pullback operators for the sake of simplicity in the notation.
    }
    \[
        \widehat{\omega} \coloneqq \beta + \diff\theta \wedge \alpha.
    \]
    In particular, if $M$ is compact, then $\widehat{M}$ is compact and the embedding
    \begin{equation} \label{eq:cosymplectic-inclusion}
        \begin{array}{rccc}
            \imath_{\theta_0} \colon & M & \xhookrightarrow{} & M \times \mathbb{S}^1\\
             & x & \longmapsto & (x, \theta_0)
        \end{array}
    \end{equation}
    for any fixed $\theta_0 \in \mathbb{S}^1$ realizes $M$ as a compact hypersurface in the symplectic manifold $(M \times \mathbb{S}^1,\widehat{\omega})$.
\end{proposition}

\begin{proof}
    Since $\diff\beta=0$ and $\diff\alpha=0$, and since $\diff(\diff\theta)=0$, we have $\diff\widehat{\omega}=0$. Thus $\widehat{\omega}$ is closed. To prove non-degeneracy, we compute
    \[
        \widehat{\omega}^{n+1}= (\beta + \diff \theta\wedge \alpha)^{n+1} = \beta^{n+1}+(n+1)\beta^n\wedge\diff\theta\wedge\alpha.
    \]
    Now $\beta^{n+1}=0$ on $M\times\mathbb{S}^1$, since $\beta$ has rank at most $2n$ and is pulled back from $M$. Hence
    \begin{equation} \label{eq:thicken-sympl-vol}
        \widehat{\omega}^{n+1} = (n+1)\diff\theta\wedge\alpha\wedge\beta^n.
    \end{equation}
    Since $\alpha\wedge\beta^n$ is a volume form on $M$ and $\diff\theta$ is nowhere vanishing on $\mathbb{S}^1$, it follows that $\widehat{\omega}^{n+1}$ is a nowhere vanishing $(2n+2)$-form on $M\times\mathbb{S}^1$. Therefore $\widehat{\omega}$ is non-degenerate and hence symplectic.
\end{proof}

\begin{remark} \label{rmk:cosymp-struct-from-symp-thick}
    For any fixed $\theta_0\in\mathbb{S}^1$, the inclusion \eqref{eq:cosymplectic-inclusion} satisfies
    \begin{equation} \label{eq:sympl-form-from-thickening}
        \imath_{\theta_0}^* \widehat{\omega}=\beta.
    \end{equation}
    Moreover, if $\partial_\theta$ denotes the vector field generated by the $\mathbb{S}^1$-factor, then
    \begin{equation} \label{eq:def-form-from-thickening}
        \iota_{\partial_\theta} \widehat{\omega} = \alpha.
    \end{equation}
    Thus the cosymplectic structure $(\alpha,\beta)$ on $M$ is recovered from the ambient symplectic form $\omega$ and the distinguished transverse direction $\partial_\theta$.

    Moreover, under the isomorphism $\mathrm{H}^\bullet(M) \otimes \mathrm{H}^\bullet(\mathbb{S}^1) \cong \mathrm{H}^\bullet(M) \oplus \mathrm{H}^{\bullet - 1}(M)$ these maps provide a concrete realization of the Künneth isomorphism
    \begin{equation*}
        \begin{array}{ccc}
            \mathrm{H}^\bullet(M \times \mathbb{S}^1) & \longrightarrow & \mathrm{H}^\bullet(M) \oplus \mathrm{H}^{\bullet - 1}(M) \\[1pt]
            [\omega] & \longmapsto & [\imath_{\theta_0}^* \omega] \oplus [\iota_{\partial_\theta} \omega]
        \end{array}.
    \end{equation*}
    As a consequence, formulae \eqref{eq:sympl-form-from-thickening} and \eqref{eq:def-form-from-thickening} also recover the cohomology classes $[\alpha], [\beta]$ from $[\omega]$.
\end{remark}

The symplectic thickening procedure admits a straightforward generalization when considering equivariant cosymplectic structures.

\begin{proposition} \label{prop:equivariant-symplectic-thickening}
    Let $(\underline{\alpha}, \underline{\beta})$ be an equivariant cosymplectic structure in $M$. If we extend the $G$-action in $M$ to the symplectic thickening $M \times \mathbb{S}^1$ as
    \begin{equation} \label{eq:cosymplectic-action-extension}
        \begin{array}{ccc}
            G \times M \times \mathbb{S}^1 & \longrightarrow & M \times \mathbb{S}^1 \\
            (g, x, \theta) & \longmapsto & (g \cdot x, \theta)
        \end{array},
    \end{equation}
    the form $\underline{\omega} = \underline{\beta} + \diff \theta \wedge \underline{\alpha}$ is an equivariantly closed symplectic form.
\end{proposition}

\begin{proof}
    It is clear from expression \eqref{eq:cosymplectic-action-extension} that the map defines a group action. Moreover, for every $X \in \mathfrak{g}$ the fundamental vector field is given by $(X^\#, 0)$.
    
    To check that the action is Hamiltonian it suffices to prove that the equivariant form $\underline{\omega} = \underline{\beta} + \diff \theta \wedge \underline{\alpha} = \beta + \diff \theta \wedge \alpha - \operatorname{pr}_1^*\mu$ is equivariantly closed. However, since
    \begin{equation*}
        \diff_G (\diff \theta) = \diff(\diff \theta) - \sum_{i = 1}^m (\iota_{(X_i^\#, 0)} \diff \theta) \otimes \sigma^i = 0
    \end{equation*}
    and $\diff_G \underline{\alpha} = \diff_G \underline{\beta} = 0$ by assumption, we have
    \begin{equation*}
        \diff_G \underline{\omega} = \diff_G \underline{\beta} + (\diff_G \diff \theta) \wedge \underline{\alpha} - \diff \theta \wedge \diff_G \underline{\alpha} = 0. \qedhere
    \end{equation*}
\end{proof}

Recall, however, that there exists a correspondence between equivariantly closed extensions of a symplectic structure and Hamiltonian $G$-actions (cf. \Cref{ex:equiv-symp-form}). Similarly, \Cref{thm:leafwise-equiv-closed-characterisation} an analogous correspondence between equivariant cosymplectic structures and cosymplectic Hamiltonian $G$-actions. Since the equivariant symplectic form is given by
\begin{equation*}
    \underline{\omega} = \underline{\beta} + \diff \theta \wedge \underline{\alpha} = (\beta + \diff \theta \wedge \alpha) - \mu = \omega - \mu,
\end{equation*}
we have the following

\begin{corollary} \label{cor:symp-thick-ham-ext}
    Let $(M^{2n+1},\alpha,\beta)$ be a cosymplectic manifold and consider its symplectic thickening $(M \times \mathbb{S}^1, \beta + \diff \theta \wedge \alpha)$. If $G$ acts in Hamiltonian fashion on $M$ with moment map $\mu \colon M \to \mathfrak{g}^*$, then the extended group action \eqref{eq:cosymplectic-action-extension} is Hamiltonian with moment map $\hat{\mu}=\operatorname{pr}_1^*\mu$.
\end{corollary}

The compatibility between cosymplectic Hamiltonian actions and compact symplectic thickening extends to reduction. This makes precise the idea that symplectic reduction of the thickened Hamiltonian space is exactly the compact symplectic thickening of the cosymplectic reduced space.

\begin{theorem} \label{thm:reduction-commutes-thickening}
    Let $\xi\in\mathfrak{g}^*$ be a regular value of $\mu$, and suppose that the
    cosymplectic reduced space
    \[
       M_\xi=\mu^{-1}(\xi)/G_\xi
    \]
    is smooth. Then $\xi$ is also a regular value of $\hat{\mu}$, and the symplectic reduction of the compact symplectic thickening satisfies
    \[
       \widehat{M}_\xi \coloneqq \widehat\mu^{-1}(\xi)/G_\xi \cong M_\xi \times \mathbb{S}^1.
    \]
    Moreover, under this identification the reduced symplectic form is
    \begin{equation} \label{eq:red-sympl-str}
        \widehat\omega_\xi=\beta_\xi+ \diff \theta\wedge\alpha_\xi,
    \end{equation}
    where $(\alpha_\xi,\beta_\xi)$ is the cosymplectic structure on $M_\xi$ obtained by cosymplectic reduction.
\end{theorem}

\begin{proof}
    Since $\hat{\mu} = \operatorname{pr}_1^* \mu$ from \Cref{cor:symp-thick-ham-ext}, we have $\hat{\mu}^{-1}(\xi)=\mu^{-1}(\xi) \times \mathbb{S}^1$. Since the action of $G$ is trivial in the second factor, and hence that of $G_\xi$, we have the following commutative diagram
    \begin{equation} \label{eq:triv-red-thick}
        \begin{tikzcd}
        	{\hat{\mu}^{-1}(\xi)} & {\mu^{-1}(\xi) \times \mathbb{S}^1} \\
        	{\widehat{M}_{\xi}} & {M_\xi \times \mathbb{S}^1}
        	\arrow[from=1-1, to=1-2]
        	\arrow[from=1-1, to=2-1]
        	\arrow[from=1-2, to=2-2]
        	\arrow[from=2-1, to=2-2]
        \end{tikzcd}
    \end{equation}
    This proves the isomorphism $\widehat{M}_\xi \cong M_\xi \times \mathbb{S}^1$ as smooth manifolds.

    To show that the Marsden--Weinstein symplectic form is given by \eqref{eq:red-sympl-str}, let us denote by $j \colon \mu^{-1}(\xi)\hookrightarrow M$ and $\hat{\jmath} \colon \hat{\mu}^{-1}(\xi) \hookrightarrow \widehat M$ the inclusions, and let $\pi \colon \mu^{-1}(\xi)\to M_\xi$ and $\hat{\pi} \colon \hat\mu^{-1}(\xi) \to \widehat{M}_\xi$ be the quotient maps. By the definition of cosymplectic reduction (cf. \Cref{thm:cos-reduction}), the reduced cosymplectic structure is uniquely determined by the equations $j^*\alpha=\pi^*\alpha_\xi$ and $j^*\beta=\pi^*\beta_\xi$. Therefore, the corresponding symplectic thickening satisfies
    \[
       \hat{\jmath}^*\widehat\omega
       = \hat{\jmath}^*(\beta + \diff \theta \wedge \alpha)
       = \pi^*\beta_\xi + \diff \theta\wedge\pi^*\alpha_\xi
       =\hat{\pi}^*(\beta_\xi+ \diff \theta\wedge\alpha_\xi).
    \]
    However, since the reduced symplectic form is unique, it follows that
    $\widehat{\omega}_\xi = \beta_\xi + \diff \theta\wedge\alpha_\xi$.
\end{proof}

\section{Torus actions in cosymplectic manifolds} \label{sec:convexity-cosymplectic}

\subsection{Convexity for torus actions}

A recurring theme in the theory of Hamiltonian actions is the \emph{convexity} of the image of the moment map for Hamiltonian actions of tori $\mathbb{T}^k$ in symplectic manifolds $(M,\omega)$. This result was originally proved by Atiyah~\cite{atiyah_convexity_1982} and Guillemin and Sternberg~\cite{guillemin_convexity_1982} and has important implications in the theory of Hamiltonian actions. It laid the groundwork for the outstanding result of Delzant~\cite{delzant_hamiltoniens_1988} characterizing Hamiltonian actions of tori in terms of so-called Delzant polytopes.

In this section we prove that, under the classical assumptions in the symplectic counterpart, the convexity theorem holds for Hamiltonian actions of tori on compact cosymplectic manifolds. We remark that this result was previously obtained by He \cite[Thm.\,4.3.5]{he_2010} and also observed by Bazzoni and Goertsches \cite[p.\,2]{bazzoni_toric_2019}. We reproduce it here to show how the procedure of symplectic thickening introduced in \Cref{prop:equivariant-symplectic-thickening} yields an immediate proof from its symplectic counterpart.

\begin{theorem}
    Let $(M,\alpha,\beta)$ be a compact cosymplectic manifold together with a Hamiltonian action of a torus $\mathbb{T}^k$. Then the image of the moment map $\mu(M) \subseteq \mathfrak{t}^* \cong \mathbb{R}^k$ is a convex polytope.
\end{theorem}

\begin{proof}
    By \Cref{prop:equivariant-symplectic-thickening}, the extended action on $(M\times\mathbb{S}^1, \beta + \diff\theta \wedge \alpha)$ is a Hamiltonian $\mathbb{T}^k$-action with moment map $\hat{\mu} = \operatorname{pr}_1^*\mu$. By the classical convexity theorem, $\hat{\mu}(M\times\mathbb S^1)\subseteq\mathfrak{t}^*$ is a convex polytope. Since $\hat{\mu}=\operatorname{pr}_1^*\mu$, we have
    \[
        \hat{\mu}(M\times\mathbb S^1) = \mu(M),
    \]
    and the result follows.
\end{proof}

\subsection{A cosymplectic Delzant theorem}

In the symplectic setting, Delzant's theorem gives a complete classification of toric symplectic manifolds in terms of their corresponding moment polytope.

\begin{theorem}[{Delzant~\cite{delzant_hamiltoniens_1988}}] \label{thm:delzant}
    The assignment
    \[
       (L^{2n},\omega_L,\mathbb{T}^n,\mu_L)\longmapsto \mu_L(L)\subset\mathfrak t^*
    \]
    induces a bijection between equivariant symplectomorphism classes of compact connected symplectic toric manifolds, with moment maps considered up to the addition of constants, and Delzant polytopes in \(\mathfrak{t}^*\) considered up to translation.
\end{theorem}

In this section we build on Delzant's theorem and on the description of toric cosymplectic manifolds due to Bazzoni and Goertsches \cite{bazzoni_toric_2019} in order to prove a cosymplectic analogue of Delzant's classification theorem. At the end of the section, we explain the relation with the \(b\)-symplectic Delzant theorem in~\cite{guillemin_toric_2015}.

\begin{definition} \label{def:cosymplectic-toric}
    Let \((M^{2n+1},\alpha,\beta)\) be a compact connected cosymplectic manifold. A Hamiltonian action of a torus \(\mathbb{T}^n\) on \(M\) is called \emph{toric} if it is effective and preserves the cosymplectic structure, that is, if
    \[
       g^*\alpha=\alpha,
       \qquad
       g^*\beta=\beta,
       \qquad
       \text{for every } g\in \mathbb{T}^n.
    \]
    A compact connected cosymplectic manifold endowed with such an action will be called a \emph{cosymplectic toric manifold}.
\end{definition}

Thus the acting torus has dimension equal to half the dimension of the symplectic leaves. The Reeb direction is not part of the torus action; rather, the torus acts leafwise and in Hamiltonian fashion.

In the proof of the cosymplectic Delzant theorem we shall use the following structure result due to Bazzoni and Goertsches~\cite[Cor.\,3.6]{bazzoni_toric_2019}, which allows us to restrict to compact mapping tori, as in \Cref{thm:mapping-torus}.

\begin{theorem} \label{thm:cpt-leaf}
    Let \((M^{2n+1},\alpha,\beta)\) be a compact cosymplectic toric manifold. Then \(M\) has a compact symplectic leaf.
\end{theorem}

Under these assumptions, \Cref{thm:mapping-torus} applies. In the toric case, one can say more about the smooth topology of \(M\).

\begin{theorem} \label{thm:tor-cosym-trivial}
    A compact toric cosymplectic manifold is, as a smooth manifold, diffeomorphic to the product of a compact symplectic toric manifold with a circle.
\end{theorem}

This fact was originally obtained in the context of toric \(b\)-symplectic manifolds in \cite[Cor.\,16]{guillemin_toric_2015} where the cosymplectic structure appears as critical set of the \(b\)-symplectic structure. In the purely cosymplectic setting, it follows from the work of Bazzoni and Goertsches~\cite[Cor.\,3.8]{bazzoni_toric_2019}.

\begin{remark} \label{rmk:not-triv-cosympl}
    The conclusion of \Cref{thm:tor-cosym-trivial} is a statement about the smooth topology of \(M\). It does not imply, in general, that \(M\) is cosymplectomorphic to the product \(L\times \mathbb S^1\) endowed with the product cosymplectic structure
    \[
       \alpha=\operatorname{pr}_2^* \diff t,
       \qquad
       \beta=\operatorname{pr}_1^*\omega_L.
    \]
    Indeed, if \(M\) is described as the mapping torus of a symplectomorphism $\varphi\colon L \to L$, the cosymplectic structure remembers the holonomy of the Reeb flow.

    In the toric situation, one can choose \(X\in\mathfrak t\) such that the time-\(c\) flow of \(R+X^\#\) restricts to the identity on \(L\), where \(c=\operatorname{mp}(\alpha,\beta)\). The torus action preserves \(\alpha\) and \(\beta\), its fundamental vector fields commute with the Reeb vector field. Indeed, since $\mathcal{L}_{X^\#}\alpha = 0, \mathcal{L}_{X^\#} \beta = 0$ and the Reeb field is characterized by $\iota_R\alpha=1, \iota_R\beta=0$, it follows that \([X^\#,R]\) satisfies $\alpha([X^\#,R])=0$ and $\iota_{[X^\#,R]} \beta=0$. By the non-degeneracy condition of the cosymplectic structure, this implies $[X^\#,R]=0$.
    
    Consequently, we have
    \[
       \varphi\circ \Phi^c_{X^\#}=\operatorname{id}_L.
    \]
    In particular, \(\varphi\) is isotopic to the identity through \(\mathbb{T}^n\)-equivariant symplectomorphisms. This is compatible with the approach of Bazzoni and Goertsches~\cite[Cor.\,3.8]{bazzoni_toric_2019}, where the isotopy to the identity is obtained using a result of Pinsonnault.

    Thus the monodromy does not affect the smooth product structure, but it remains part of the cosymplectic data. In other words, \Cref{thm:tor-cosym-trivial} trivializes the mapping torus as a smooth manifold, whereas the cosymplectic classification still records the Reeb holonomy.
\end{remark}

We are now in position to prove the cosymplectic analogue of Delzant's classification theorem.

\begin{theorem}[Cosymplectic Delzant classification] \label{thm:cosymplectic-delzant}
    Compact connected cosymplectic toric manifolds of dimension \(2n+1\) are classified by triples
    \[
       (\Delta,[\varphi],c),
    \]
    where
    \begin{itemize}
        \item \(\Delta\subset\mathfrak{t}^*\) is a Delzant polytope, considered up to translation;
        \item \(c>0\) is the modular period;
        \item \([\varphi]\) is the conjugacy class, under \(\mathbb{T}^n\)-equivariant symplectomorphisms of \(L_\Delta\), of a \(\mathbb{T}^n\)-equivariant symplectomorphism
        \[
           \varphi\colon L_\Delta\longrightarrow L_\Delta
        \]
        of the compact symplectic toric manifold
        \[
           (L_\Delta,\omega_\Delta,\mathbb{T}^n,\mu_\Delta)
        \]
        associated with \(\Delta\).
    \end{itemize}
    Here \(L_\Delta\) denotes any compact symplectic toric manifold associated with a representative of the translation class of \(\Delta\). After choosing a representative moment map \(\mu_\Delta\), the condition
    \[
       \mu_\Delta\circ\varphi=\mu_\Delta
    \]
    is automatic: any \(\mathbb{T}^n\)-equivariant symplectomorphism of the compact symplectic toric manifold \((L_\Delta,\omega_\Delta,\mathbb{T}^n,\mu_\Delta)\) preserves \(\mu_\Delta\).

    The cosymplectic toric manifold associated with this data is the mapping torus
    \[
       M_{\Delta,\varphi,c}
       =
       \frac{L_\Delta\times[0,c]}{(x,0)\sim(\varphi(x),c)},
    \]
    endowed with the cosymplectic structure induced by \(\diff t\) and \(\omega_\Delta\).
\end{theorem}

\begin{proof}
    We first construct a cosymplectic toric manifold from the given data. By Delzant's theorem, the Delzant polytope \(\Delta\subset\mathfrak{t}^*\), once a representative up to translation has been fixed, determines a compact symplectic toric manifold $(L_\Delta,\omega_\Delta,\mathbb{T}^n,\mu_\Delta)$, uniquely up to \(\mathbb{T}^n\)-equivariant symplectomorphism.

    Let $\varphi\colon L_\Delta\longrightarrow L_\Delta$ be a \(\mathbb{T}^n\)-equivariant symplectomorphism. We first show that \(\varphi\) automatically preserves the chosen moment map \(\mu_\Delta\). We use the convention $\diff\mu_\Delta^X=-\iota_{X^\#}\omega_\Delta$, $X\in\mathfrak{t}$. Since \(\varphi\) is \(\mathbb{T}^n\)-equivariant, it satisfies $\varphi_*X^\#=X^\#$. Since \(\varphi\) is also symplectic, \(\varphi^*\omega_\Delta=\omega_\Delta\). Therefore $\varphi^*\iota_{X^\#}\omega_\Delta = \iota_{X^\#}\omega_\Delta.$ Hence, for every \(X\in\mathfrak{t}\),
    \[
       \diff(\mu_\Delta^X\circ\varphi)
       =
       \varphi^*\diff\mu_\Delta^X
       =
       -\varphi^*\iota_{X^\#}\omega_\Delta
       =
       -\iota_{X^\#}\omega_\Delta
       =
       \diff\mu_\Delta^X.
    \]
    Thus \(\mu_\Delta\circ\varphi-\mu_\Delta\) is constant, say equal to \(a\in\mathfrak{t}^*\). Since \(\varphi\) is a diffeomorphism of \(L_\Delta\), this constant translation preserves the image polytope:
    \[
       \Delta
       =
       \mu_\Delta(L_\Delta)
       =
       (\mu_\Delta\circ\varphi)(L_\Delta)
       =
       \Delta+a.
    \]
    Since \(\Delta\) is compact, this forces \(a=0\). Therefore $\mu_\Delta\circ\varphi=\mu_\Delta$.

    We define
    \[
       M_{\Delta,\varphi,c}
       =
       \frac{L_\Delta\times[0,c]}{(x,0)\sim(\varphi(x),c)}.
    \]
    Since \(\varphi^*\omega_\Delta=\omega_\Delta\), the two-form \(\omega_\Delta\) descends to a closed two-form \(\beta\) on \(M_{\Delta,\varphi,c}\). The one-form \(\diff t\) also descends to a closed one-form \(\alpha\). Moreover, $\alpha\wedge\beta^n$ is a volume form, and therefore \((\alpha,\beta)\) is a cosymplectic structure. The modular period is \(c\).

    The \(\mathbb{T}^n\)-action on \(L_\Delta\) descends to \(M_{\Delta,\varphi,c}\), because \(\varphi\) commutes with the \(\mathbb{T}^n\)-action. It remains effective, since its restriction to each symplectic leaf is effective. Since \(\mu_\Delta\circ\varphi=\mu_\Delta\), the moment map \(\mu_\Delta\) descends to the quotient. Hence the descended \(\mathbb{T}^n\)-action is Hamiltonian, and \(M_{\Delta,\varphi,c}\) is a cosymplectic toric manifold.

    Conversely, let \((M^{2n+1},\alpha,\beta)\) be a compact connected cosymplectic toric manifold. By \Cref{thm:cpt-leaf} and \Cref{thm:mapping-torus}, \(M\) is cosymplectomorphic to the mapping torus of a compact symplectic leaf \(L\), with modular period $c=\operatorname{mp}(\alpha,\beta)$ and monodromy $\varphi=\Phi_R^c|_L\colon L\to L$. The restriction of the \(\mathbb{T}^n\)-action to \(L\) is an effective Hamiltonian torus action on the compact connected symplectic manifold \((L,\beta|_L)\). Hence, by Delzant's theorem, it determines a Delzant polytope $\Delta=\mu(L)\subset\mathfrak{t}^*$, well defined up to translation.

    Since the torus action preserves \(\alpha\) and \(\beta\), its fundamental vector fields commute with the Reeb vector field (cf. \ref{rmk:not-triv-cosympl}). Therefore the monodromy \(\varphi\) is \(\mathbb{T}^n\)-equivariant. Moreover, \(\varphi\) is a symplectomorphism of \((L,\beta|_L)\), because it is obtained as the time-\(c\) map of the Reeb flow, and the Reeb flow preserves \(\beta\). Hence \(\varphi\) defines the required conjugacy class \([\varphi]\).

    Finally, if two compact connected cosymplectic toric manifolds are \(\mathbb{T}^n\)-equivariantly cosymplectomorphic, then their modular periods agree, their Delzant polytopes agree up to translation, and their monodromies are conjugate by a \(\mathbb{T}^n\)-equivariant symplectomorphism of the corresponding symplectic leaves. Conversely, such a conjugacy induces a \(\mathbb{T}^n\)-equivariant cosymplectomorphism between the corresponding mapping tori. Hence the two constructions are mutually inverse.
\end{proof}

\begin{corollary} \label{cor:smooth-equivariant-classification}
    Compact connected cosymplectic toric manifolds are, up to \(\mathbb{T}^n\)-equivariant diffeomorphism, products
    \[
       L_\Delta\times \mathbb S^1.
    \]
    In particular, their smooth equivariant type is determined by the Delzant polytope \(\Delta\), considered up to translation.
\end{corollary}

\begin{remark} \label{rmk:monodromy-needed}
    The monodromy \(\varphi\) is not needed in the smooth equivariant classification of \Cref{cor:smooth-equivariant-classification}, but it is needed in the cosymplectic classification of \Cref{thm:cosymplectic-delzant}. Indeed, the mapping torus may be smoothly trivial even when the cosymplectic structure has non-trivial Reeb holonomy. Thus the data $(\Delta,[\varphi],c)$ records the Delzant type of the symplectic leaves, the return map of the Reeb flow, and the modular period. The polytope alone records only the equivariant symplectic type of the leaves.
\end{remark}

\begin{remark} \label{rmk:b-symplectic-delzant}
    The preceding classification is closely related to the Delzant theorem for toric \(b\)-symplectic manifolds. We briefly indicate the relation, without using it as a substitute for the proof above.

    The construction below is a version of the realization of cosymplectic manifolds as critical hypersurfaces of \(b\)-symplectic manifolds, already present in the work of Guillemin, Miranda and Pires in~\cite{guillemin_codimension_2011}, and  in~\cite{guillemin_symplectic_2014}. Let \((M,\alpha,\beta)\) be a compact cosymplectic manifold and let $f\colon \mathbb{S}^1 \to \mathbb{R}$ be a smooth function such that \(0\) is a regular value of \(f\). Set $\widetilde{M} = M \times \mathbb{S}^1, \widetilde{Z} = M \times f^{-1}(0)$.
    On \(\widetilde{M}\), consider the \(b\)-two-form
    \[
       \widetilde{\omega}
       =
       \beta+\frac{\diff \theta}{f(\theta)}\wedge\alpha.
    \]
    Since \(0\) is a regular value of \(f\), the hypersurface \(\widetilde{Z}\) is smooth, possibly with several connected components. The one-form \(\diff \theta/f(\theta)\) is a \(b\)-one-form with logarithmic singularities along \(\widetilde{Z}\). Since \(\diff \alpha=0\) and \(\diff \beta=0\), the form \(\widetilde{\omega}\) is closed as a \(b\)-form. Moreover,
    \[
       \widetilde{\omega}^{n+1}
       =
       (n+1)\,\beta^n\wedge \frac{\diff \theta}{f(\theta)}\wedge\alpha
    \]
    is a nowhere-vanishing \(b\)-volume form. Hence $(\widetilde{M},\widetilde{Z},\widetilde{\omega})$ is a \(b\)-symplectic manifold.

    Along a connected component $\widetilde{Z}_{\theta_0}=M\times\{\theta_0\}, f(\theta_0)=0$,
    the induced cosymplectic structure is given, up to the sign and scaling determined by the choice of local defining function, by
    \[
       \alpha_{\theta_0}
       =
       \frac{1}{f'(\theta_0)}\alpha,
       \qquad
       \beta_{\theta_0}
       =
       \beta.
    \]
    Equivalently, if one uses the local defining function $x = f(\theta)$, then near \(\theta_0\) one has
    \[
       \frac{\diff \theta}{f(\theta)}
       =
       \frac{1}{f'(\theta_0)}\frac{\diff x}{x}
       +
       \eta,
    \]
    where \(\eta\) is a smooth one-form. Thus the residue of \(\widetilde{\omega}\) along \(\widetilde{Z}_{\theta_0}\) is $\frac{1}{f'(\theta_0)}\alpha$, and the original cosymplectic structure \((\alpha,\beta)\) is recovered up to the standard normalization of the defining function.

    If \(R\) denotes the Reeb vector field of \((\alpha,\beta)\), then
    \[
       \iota_R\widetilde\omega
       =
       -\frac{\diff \theta}{f(\theta)}.
    \]
    Locally, near a zero \(\theta_0\) of \(f\), this is a \(b\)-one-form with logarithmic singularity along \(\widetilde{Z}_{\theta_0}\).

    In the toric situation, the periodic vector field relevant to the additional circle direction is generally not \(R\) itself, but a vector field of the form $R+X^\#, X \in \mathfrak{t}$,
    as explained in \Cref{rmk:not-triv-cosympl}. In this way, the cosymplectic toric picture can be regarded as the critical hypersurface model underlying the corresponding \(b\)-symplectic toric classification. The \(b\)-moment image should be understood in the \(b\)-affine sense of \cite{guillemin_toric_2015}, rather than as an ordinary polytope in \(\mathfrak{t}^*\times\mathbb{S}^1\).
\end{remark}

\section{Localization formulae for cosymplectic manifolds} \label{sec:localization-cosymplectic}

Localization is a general phenomenon which allows one to compute the integral of an equivariantly closed form as a sum of certain quantities over the fixed-point set of the $G$-action. This result is not intrinsic to symplectic geometry: an abstract approach to localization was given by Atiyah and Bott~\cite[Thm.~3.5, eq.~3.8]{atiyah-momentmap-1984}. Previous instances of the localization formula were obtained by Berline and Vergne~\cite{berline-classescaracteristiques-1982,berline_zeros_1983} for isolated zeros of a group action. The localization formula in general is given by

\begin{theorem} \label{thm:equivariant-localization}
    Let $M$ be a compact manifold together with an action of a compact Lie group $G$. If there exists $X \in \mathfrak{g}$ such that $X^\#$ is a non-degenerate vector field, then for every equivariantly closed form $\underline{\tau} \in \Omega^{\dim M}_G(M)$ we have
    \begin{equation}
        \int_M \underline{\tau} = \sum_{C \in \pi_0(M^G)} \int_C \frac{i_C^* \underline{\tau}}{e_G(\nu_C)}.
    \end{equation}
\end{theorem}

In the symplectic setting, these localization formulae were used by Berline and Vergne to give a conceptually simple and elegant proof of the Duistermaat--Heckman stationary phase formula~\cite[Thm.~4.1]{duistermaat-variationcohomology-1982}. The proof follows from an application of \Cref{thm:equivariant-localization} to the equivariantly closed symplectic form $\mathrm{e}^{\underline{\omega}} = \mathrm{e}^{\omega - \mu}$.

In this section we use the symplectic thickening technique to obtain localization formulae for Hamiltonian group actions in the cosymplectic setting. As a byproduct, we obtain localization formulae for circle actions which resemble stationary phase formulae. However, in the cosymplectic setting the fixed points form one-dimensional embedded submanifolds and the localization formula contains additional information related to the cosymplectic structure.

As in the previous sections, let us fix a cosymplectic manifold $(M,\alpha,\beta)$ with a Hamiltonian action of a compact Lie group $G$. Let $\underline{\tau}\in\Omega_G^\bullet(M)$ be an arbitrary equivariantly closed form and consider its extension to the symplectic thickening $M\times\mathbb{S}^1$ by pullback along the first component. As in the classical localization formula for symplectic manifolds, we consider the equivariantly closed symplectic form $\mathrm{e}^{\underline{\omega}}=\mathrm{e}^{\omega-\mu}$. Given
\[
    \underline{\omega} = \underline{\beta} + \diff\theta \wedge \underline{\alpha} =\beta-\mu+\diff\theta\wedge\alpha,
\]
and since the forms $\beta-\mu$ and $\diff\theta\wedge\alpha$ commute under the exterior product, we have
\[
    \mathrm{e}^{\underline{\omega}}
    =\mathrm{e}^{\beta-\mu}\,\mathrm{e}^{\diff\theta\wedge\alpha}.
\]
Since $(\diff\theta\wedge\alpha)^2=0$, we have $\mathrm{e}^{\diff\theta\wedge\alpha}=1+\diff\theta\wedge\alpha$. Consequently,
\[
    \underline{\tau}\,\mathrm{e}^{\underline{\omega}} = \underline{\tau}\,\mathrm{e}^{\beta-\mu}(1+\diff\theta\wedge\alpha).
\]
Therefore,
\begin{equation*}
    \int_{M \times \mathbb{S}^1} \underline{\tau} \, \mathrm{e}^{\underline{\omega}}
    = \int_{M \times \mathbb{S}^1} \underline{\tau} \, \mathrm{e}^{\beta - \mu}
    + \int_{M \times \mathbb{S}^1} \underline{\tau} \, \mathrm{e}^{\beta - \mu} \wedge \diff \theta \wedge \alpha.
\end{equation*}

The first integral vanishes, because its integrand is pulled back from $M$ and therefore has degree at most $2n+1$ in the $M$-directions. Hence
\[
    \int_{M\times\mathbb S^1}\underline{\tau} \, \mathrm{e}^{\beta-\mu}=0.
\]
For the second term, Fubini's theorem gives
\begin{equation*}
    \int_{M\times\mathbb{S}^1}\underline{\tau}\,\mathrm{e}^{\beta-\mu}\wedge\diff\theta\wedge\alpha
    = - \int_{\mathbb{S}^1} \diff\theta \int_M \underline{\tau} \,\mathrm{e}^{\beta-\mu}\wedge\alpha
    = - \int_M \underline{\tau} \,\mathrm e^{\beta-\mu}\wedge\alpha.
\end{equation*}

The fixed-point set satisfies $(M\times\mathbb{S}^1)^G=M^G\times\mathbb{S}^1$. Consequently, each connected component $\widehat C\in\pi_0(M^G\times\mathbb S^1)$ is of the form $\widehat{C}=C\times\mathbb S^1$ for some $C\in\pi_0(M^G)$. This relation implies $\nu_{\widehat C}=\operatorname{pr}_1^*\nu_C$ and, by functoriality of the Euler class, $e_G(\nu_{\widehat C})=\operatorname{pr}_1^*e_G(\nu_C)$.

The Atiyah--Bott--Berline--Vergne localization formula reads
\begin{equation*}
    \int_{M \times \mathbb{S}^1} \underline{\tau} \, \mathrm{e}^{\underline{\omega}}
    = \sum_{C \in \pi_0(M^G)}
    \int_{C \times \mathbb{S}^1}
    \frac{i_{\widehat C}^*(\underline{\tau}\,\mathrm{e}^{\underline{\omega}})}{e_G(\nu_{\widehat C})}.
\end{equation*}
Moreover,
\[
 i_{\widehat C}^*(\underline{\tau}\,\mathrm{e}^{\underline{\omega}})
 =\diff\theta\wedge i_C^*(\underline{\tau}\,\mathrm e^{\beta-\mu}\alpha),
\]
and another application of Fubini's theorem shows that
\begin{equation*}
    \int_{C\times\mathbb{S}^1}
    \frac{i_{\widehat C}^*(\underline{\tau}\,\mathrm{e}^{\underline{\omega}})}{e_G(\nu_{\widehat C})}
    = - \int_C\frac{i_C^*(\underline{\tau}\,\mathrm{e}^{\beta-\mu}\alpha)}{e_G(\nu_C)}.
\end{equation*}

Putting everything together, we obtain the following result.

\begin{theorem}[Localization formula for cosymplectic manifolds] \label{thm:cosymplectic-localization}
    Let $(M, \alpha, \beta)$ be a compact cosymplectic manifold with a Hamiltonian action of a compact connected Lie group $G$. Then for every equivariantly closed form $\underline{\tau} \in \Omega_G(M)$,
    \begin{equation} \label{eq:cosymplectic-localization}
        \int_M \alpha \wedge \underline{\tau} \, \mathrm{e}^{\beta - \mu}
        =
        \sum_{C \in \pi_0(M^G)}
        \int_C
        \alpha \wedge
        \frac{i_C^*\big(\underline{\tau} \, \mathrm{e}^{\beta - \mu}\big)}{e_G(\nu_C)}.
    \end{equation}
\end{theorem}

\begin{remark}
    Even though the result could be directly proved from the general localization formula applied to the equivariantly closed form $\underline{\tau}\,\mathrm{e}^{\beta-\mu} \, \alpha$, the approach following the classical localization formula for the symplecticzation provides a clear conceptual explanation for the appearance of $\alpha$ in expression \eqref{eq:cosymplectic-localization}.
\end{remark}

\begin{example} \label{ex:localization-trivial-tau}
    Under the assumptions of \Cref{thm:cosymplectic-localization}, take $\underline{\tau}=1\in\mathrm{H}_G^0(M)$. Then the localization formula \eqref{eq:cosymplectic-localization} reads
    \begin{equation*}
        \int_M \alpha \,\mathrm{e}^{\beta-\mu}
        =\sum_{C\in\pi_0(M^G)}\int_C\alpha\,\frac{\mathrm{e}^{i_C^*(\beta-\mu)}}{e_G(\nu_C)}.
    \end{equation*}
    Expanding and keeping only the component of degree $2n+1$, we obtain
    \begin{equation} \label{eq:localization-trivial-tau}
        \int_M \alpha\,\frac{\beta^n}{n!}\,\mathrm{e}^{-\mu}
        =\sum_{C\in\pi_0(M^G)}\int_C\alpha\,
        \frac{\mathrm{e}^{i_C^*\beta}\,\mathrm{e}^{-\mu|_C}}{e_G(\nu_C)}.
    \end{equation}
\end{example}

\begin{example} \label{ex:cosymplectic-stationary-phase}
    As in the previous example, assume $\underline{\tau}=1$. Let $G=\mathbb{S}^1$. After choosing an identification $\mathfrak{t}\cong\mathbb{R}$, the moment map can be identified with an element $\mu = fu$, where $f\in\mathcal C^\infty(M)$ and $u$ is the generator of $\operatorname{Sym}(\mathfrak t^*)\cong\mathbb{R}[u]$.

    Assume that there exists a compact leaf $L$ and that the fixed points of the $\mathbb{S}^1$-action on $L$ are isolated. Then the connected components of $M^{\mathbb{S}^1}$ are circles generated by the Reeb flow on $M$.

    If $p\in L^{\mathbb{S}^1}$, the equivariant Euler class of the normal bundle can be computed in terms of the isotropy weights $m_i(p)$ of the $\mathbb{S}^1$-action on $\mathrm{T}_pL$ as
    \begin{equation*}
        e_{\mathbb{S}^1}(\nu_p)(u)=\Big(\prod_{i=1}^n m_i(p)\Big)u^n.
    \end{equation*}
    Under these assumptions, the localization formula \eqref{eq:cosymplectic-localization} reads
    \begin{equation*}
        \int_M\alpha\,\frac{\beta^n}{n!}\,\mathrm{e}^{-fu}
        =\sum_{p\in L^{\mathbb{S}^1}}\int_{\mathbb{S}^1}
        \frac{\mathrm{e}^{-f(p)u}}{\prod_{i=1}^n m_i(p)u^n}.
    \end{equation*}
    Setting $u=\mathrm{i} t$ gives the cosymplectic analogue of the classical stationary phase formula of Duistermaat and Heckman:
    \begin{equation*}
        \int_M\alpha\,\frac{\beta^n}{n!}\,\mathrm{e}^{-\mathrm{i} f t}
        =\sum_{p\in L^{\mathbb{S}^1}}\int_{\mathbb{S}^1}
        \frac{\mathrm{e}^{-\mathrm{i} f(p)t}}{\prod_{i=1}^n m_i(p)\,\mathrm{i}^n t^n}.
    \end{equation*}
\end{example}

\subsection{Localization formulae from mapping tori}

We have thus far described various localization formulae for Hamiltonian group actions in cosymplectic manifolds. For a cosymplectic manifold $M$ in the assumptions of \Cref{thm:mapping-torus}, such formulae admit further refinements. Let us begin by proving the following lemma which reduces certain integrals appearing in the localization formula to integrals on the leaf $L$.

\begin{lemma} \label{lem:modular-weight-integral}
    Let $M$ be a compact cosymplectic manifold with compact leaf $L$ and denote $\omega_L\coloneqq i_L^*\beta$ and $\mu_L=i_L^*\mu$. Let $W\subseteq\mathfrak g^*$ be a $\mu$-measurable subset. Then
    \begin{align}
        \int_M \alpha \wedge \frac{\beta^n}{n!} &= \operatorname{mp}(\alpha, \beta) \int_L \frac{\omega_L^n}{n!}, \label{eq:int1} \\
        \int_{\mu^{-1}(W)} \alpha \wedge \frac{\beta^n}{n!} &= \operatorname{mp}(\alpha, \beta) \int_{\mu_L^{-1}(W)} \frac{\omega_L^n}{n!}. \label{eq:int2}
    \end{align}
\end{lemma}

\begin{proof}
    Both computations boil down to the following construction. Because $L=\pi^{-1}(p)$ has zero measure in $M$, it suffices to compute the respective integrals over $M\setminus L$. From the description of $M$ as a fibre bundle over $\mathbb S^1$, the map
    \begin{equation*}
        \begin{array}{rccc}
            \zeta \colon & L \times (0,c) & \longrightarrow & M\setminus L \\
             & (q, t) & \longmapsto & \Phi_R^t(q)
        \end{array}
    \end{equation*}
    is a diffeomorphism. Moreover, since $\iota_R\alpha=1$ and $\iota_R\beta=0$, we have that $\zeta^*\alpha=\operatorname{pr}_2^*\diff t$ and $\zeta^*\beta=\operatorname{pr}_1^*\omega_L$. Since $\mathcal{L}_R\mu=0$, we also have that $\zeta^*\mu=\operatorname{pr}_1^*\mu_L$. In particular,
    \[
        \mu^{-1}(W)\cap(M\setminus L)=\zeta\big( \mu_L^{-1}(W) \times (0,c) \big).
    \]

    For expression \eqref{eq:int1}, direct computation shows
    \begin{align*}
        \int_M\alpha\wedge\frac{\beta^n}{n!} = \int_{M\setminus L}\alpha\wedge\frac{\beta^n}{n!} = \int_{L \times (0,c)}\operatorname{pr}_2^*\diff t \wedge \frac{\operatorname{pr}_1^*\omega_L^n}{n!} =\int_0^c\diff t \int_L\frac{\omega_L^n}{n!}.
    \end{align*}
    The result follows from the fact that $c=\operatorname{mp}(\alpha,\beta)$.

    The second case follows similarly:
    \begin{align*}
        \int_{\mu^{-1}(W)}\alpha\wedge\frac{\beta^n}{n!} &= \int_{\mu^{-1}(W)\cap(M\setminus L)}\alpha\wedge\frac{\beta^n}{n!} \\
        &= \int_{\mu_L^{-1}(W) \times (0,c)}\operatorname{pr}_2^*\diff t \wedge \frac{\operatorname{pr}_1^*\omega_L^n}{n!} = \int_0^c\diff t \int_{\mu_L^{-1}(W)}\frac{\omega_L^n}{n!}.\qedhere
    \end{align*}
\end{proof}

\begin{remark}
    The previous lemma unveils an interesting fact concerning measure-theoretic results when regarding $M$ as a mapping torus: because individual leaves have zero measure, we can perform integration in the open complement $M\setminus L$, which can be trivialized in terms of the flow of the Reeb vector field. We have represented this procedure in \Cref{fig:trivialization-mapping-torus}.
    \begin{figure}[h]
        \centering
        \begin{subfigure}[b]{0.4\textwidth}
            \centering
            \includegraphics[width=\textwidth]{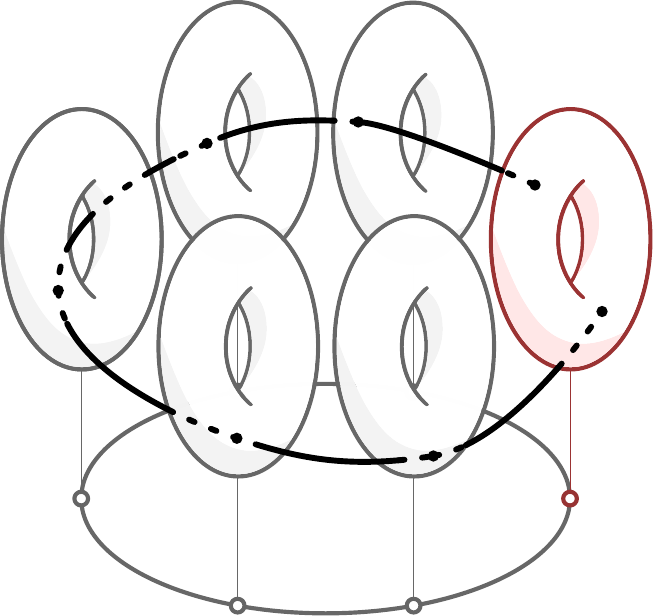}
            \caption{Mapping torus with leaf $\mathbb{T}^2$}
            \label{subfig:mapping-torus}
        \end{subfigure} \hspace{2em}
        \begin{subfigure}[b]{0.45\textwidth}
            \centering
            \raisebox{0.27\height}{\includegraphics[width=\textwidth,keepaspectratio,height=10cm]{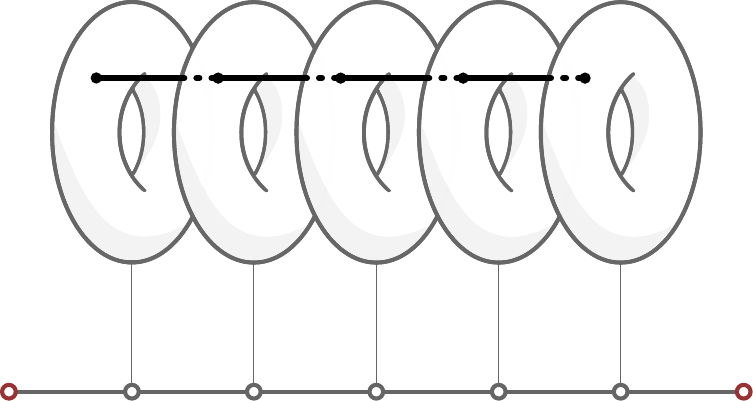}}
            \caption{Trivial fibration with leaf $\mathbb{T}^2$}
            \label{subfig:mapping-torus-developed}
        \end{subfigure}
        \caption{Transformation between a mapping torus (\Cref{subfig:mapping-torus}) and a trivial fibration (\Cref{subfig:mapping-torus-developed}) obtained by removing the red leaf. The black lines correspond to the flow of the Reeb field $R$ and its corresponding linearization.}
        \label{fig:trivialization-mapping-torus}
    \end{figure}
\end{remark}

With this result, we have the following expression for the localization formula \eqref{eq:localization-trivial-tau}. The proof follows directly from \Cref{ex:localization-trivial-tau} together with the argument used in the proof of \Cref{lem:modular-weight-integral}.

\begin{theorem}\label{thm:cosymplectic-localization-mapping}
    Let $(M,\alpha,\beta)$ be a cosymplectic manifold satisfying the mapping-torus assumptions and endowed with a Hamiltonian $G$-action. Then
    \begin{equation}\label{eq:cosymplectic-localization-mapping}
        \int_M \alpha\,\frac{\beta^n}{n!}\mathrm{e}^{-\mu}
        =
        \operatorname{mp}(\alpha,\beta)
        \sum_{C\in\pi_0(L^G)}
        \int_C
        \frac{\mathrm{e}^{\omega_L}\,\mathrm{e}^{-\mu|_C}}{e(\nu_C)}.
    \end{equation}
\end{theorem}

Applying the theorem to the circle action of \Cref{ex:cosymplectic-stationary-phase}, one obtains the following stationary phase formula in the cosymplectic setting.

\begin{theorem}\label{thm:cosymplectic-stationary-phase-mapping}
    Let $(M,\alpha,\beta)$ be a cosymplectic manifold in the assumptions of \Cref{thm:mapping-torus} and endowed with a Hamiltonian $\mathbb{S}^1$-action. Then
    \begin{equation}\label{eq:cosymplectic-localization-circle-mapping}
        \int_M \alpha\,\frac{\beta^n}{n!}\mathrm{e}^{-fu}
        =
        \operatorname{mp}(\alpha,\beta)
        \sum_{C\in\pi_0(L^{\mathbb{S}^1})}
        \int_C
        \frac{\mathrm{e}^{\omega_L}\,\mathrm{e}^{-f|_C u}}{e(\nu_C)}.
    \end{equation}
    
    In particular, if the fixed points of the $\mathbb{S}^1$-action on $L$ are isolated, then setting $u=\mathrm{i} t$ gives
    \begin{equation}\label{eq:cosymplectic-stationary-mapping}
        \int_M \alpha\,\frac{\beta^n}{n!}\mathrm{e}^{-\mathrm{i} f t}
        =
        \operatorname{mp}(\alpha,\beta)
        \sum_{p\in L^{\mathbb{S}^1}}
        \frac{\mathrm{e}^{-\mathrm{i} f(p)t}}
        {\big(\prod_{i=1}^n m_i(p)\big)\mathrm{i}^n t^n}.
    \end{equation}
\end{theorem}

\begin{remark}
    The reader should notice that the condition $\mathcal{L}_R \mu = 0$ in \eqref{eq:cos-mom-map-reeb-inv} implies that the fixed-point set of a $G$-action can never be isolated. Although the fixed points of the $\mathbb{S}^1$-action on $M$ appearing in expression \eqref{eq:cosymplectic-stationary-mapping} are not isolated, the Reeb field $R$ conjugates the action along the fibres of $\pi$, and therefore also conjugates the isotropy representations. In particular, the isotropy weights are well defined.
\end{remark}

\section{Kirwan surjectivity for cosymplectic manifolds} \label{sec:kirwan-cosymplectic}

Surjectivity in symplectic geometry originates in the seminal work of Kirwan~\cite{kirwan} on the cohomology ring of the Marsden--Weinstein reduced spaces $M\sslash G$. More precisely, Kirwan proved the following

\begin{theorem}[Kirwan~\cite{kirwan}]\label{thm:kirwan-surjectivity}
    Let $(M,\omega)$ be a compact symplectic manifold endowed with a Hamiltonian action of a compact, connected Lie group $G$. In the assumptions of the Marsden--Weinstein reduction theorem, the natural map
    \[
        \kappa\colon\mathrm{H}_G^\bullet(M)\longrightarrow\mathrm{H}^\bullet(M\sslash G),
    \]
    induced by the inclusion $i\colon\mu^{-1}(0)\hookrightarrow M$, is surjective.
\end{theorem}

We give a simple proof of Kirwan's surjectivity theorem for cosymplectic manifolds using the symplectic thickening technique developed in \Cref{sec:symplectic-thickening}. The following result is a simple corollary of \Cref{thm:reduction-commutes-thickening}, since the Kirwan maps of the spaces are induced from the inclusions $j \colon \mu^{-1}(\xi) \to M$ and $\hat{\jmath} \colon \hat{\mu}^{-1}(\xi) \to \widehat{M}$ and that they are related by $\hat{\jmath} = j \times \operatorname{id}_{\mathbb{S}^1}$.

\begin{corollary}\label{cor:kirwan-compatible-thickening}
    Under the assumptions of \Cref{thm:reduction-commutes-thickening}, the
    Kirwan map for the thickened symplectic Hamiltonian space identifies with
    \begin{equation} \label{eq:kirwan-map-thickening}
        \hat{\kappa} = \kappa \otimes \operatorname{id}_{\mathrm{H}^\bullet(\mathbb{S}^1)}
    \end{equation}
    under the Künneth isomorphisms
    \[
       \mathrm{H}_G^\bullet(M \times \mathbb{S}^1) \cong \mathrm{H}_G^\bullet(M) \otimes \mathrm{H}^\bullet(\mathbb{S}^1),
       \qquad
       \mathrm{H}^\bullet(\widehat{M}_\xi)
       \cong \mathrm{H}^\bullet(M_\xi) \otimes \mathrm{H}^\bullet(\mathbb{S}^1).
    \]
\end{corollary}

The previous result directly implies the cosymplectic version of Kirwan's surjectivity theorem in the cosymplectic setting.

\begin{theorem}
    Let $(M,\alpha,\beta)$ be a compact cosymplectic manifold endowed with a Hamiltonian action of a compact connected Lie group $G$, with moment map $\mu\colon M\to\mathfrak{g}^*$ satisfying the assumptions of \Cref{thm:cos-reduction}. Then the Kirwan map
    \[
        \kappa\colon\mathrm H_G^\bullet(M)\longrightarrow\mathrm H^\bullet(M\sslash G),
    \]
    induced by the inclusion $\mu^{-1}(0)\hookrightarrow M$, is surjective.
\end{theorem}

\begin{proof}
    Since the symplectic thickening $\widehat{M}$ is a compact symplectic manifold endowed with a Hamiltonian action of a compact, connected Lie group $G$, Kirwan's \Cref{thm:kirwan-surjectivity} applies and we have that $\hat{\kappa} = \kappa \otimes \operatorname{id}_{\mathrm{H}^\bullet(\mathbb{S}^1)}$ (cf. eq. \eqref{eq:kirwan-map-thickening}) is a surjective morphism. But recall now that, if $V$ is a vector space over a field $\mathbb{K}$, a linear map $f \colon E \to F$ is surjective if and only if $f \otimes \operatorname{id}_V \colon E \otimes_{\mathbb{K}} V \to F \otimes_{\mathbb{K}} V$ is surjective. Therefore, $\kappa$ is surjective.
\end{proof}

\section{Variation of the cohomology class of a cosymplectic structure} \label{sec:var-coh-class}

Finally, we discuss the variation of the cohomology classes of the cosymplectic structure along the reduced spaces following Duistermaat and Heckman~\cite{duistermaat-variationcohomology-1982}. The proofs use once again the procedure of symplectic thickening to apply the results from the symplectic setting. As a corollary of the results, we recover the invariance of the modular period for the reduced spaces.

Let us briefly recall the setting of Duistermaat and Heckman as in \cite{duistermaat-variationcohomology-1982}. Let $(M, \omega)$ be a symplectic manifold with a Hamiltonian action of a torus $\mathbb{T}^k$. Assume that the moment map $\mu \colon M \to \mathfrak{t}^*$ is proper and, for the sake of simplicity, that the $\mathbb{T}^k$-action is free.\footnote{
    This condition ensures that $M_\xi \coloneqq \mu^{-1}(\xi) / \mathbb{T}^k$ is a smooth manifold. In the original proof by Duistermaat and Heckman, the action is only locally free: thus, the map $\mu^{-1}(\xi) \to M_\xi$ is not a $\mathbb{T}^k$-bundle and they are forced to consider the $\mathbb{T}^k$-bundle $\mu^{-1}(\xi) / \Gamma_\xi \to M_\xi$ with $\Gamma_\xi$ the subgroup generated by all finite stabilizers $\mathbb{T}^k_p$ with $p \in \mu^{-1}(\xi)$.

    Since our proof will use Duistermaat and Heckman's result with \Cref{thm:reduction-commutes-thickening}, this assumption is irrelevant and its purpose is to simplify the presentation.
}
If $U \subset \mathfrak{t}^*$ is a convex open neighbourhood of $\xi_0 \in \mathfrak{t}^*$ consisting of regular values, the map $\mu^{-1}(U) \to U$ is, by Ehresmann's theorem, a fibre bundle with fibre $\mu^{-1}(\xi)$ over $\xi \in \mathfrak{t}^*$. One may choose a $\mathbb{T}^k$-invariant connection for this bundle and obtain, by parallel transport along the segment $\overline{\xi_0 \xi}$, an isomorphism $\mu^{-1}(\xi_0) \cong \mu^{-1}(\xi)$. Since the connection is $\mathbb{T}^k$-equivariant we have an isomorphism of bundles between $\mu^{-1}(\xi_0) \to M_{\xi_0}$ and $\mu^{-1}(\xi) \to M_\xi$. This implies that we have an identification $M_{\xi_0} \cong M_\xi$ and hence an isomorphism of cohomology groups $\mathrm{H}^\bullet(M_{\xi_0}) \cong \mathrm{H}^\bullet(M_\xi)$: this fact is what allows us to consider the variation of the symplectic forms in cohomology.

Since the \(\mathbb{T}^k\)-bundles $\mu^{-1}(\xi)\to M_\xi$ are isomorphic as \(\xi\) varies in a connected component of the set of regular values, their characteristic classes agree. We denote by $\Omega \in \mathrm{H}^2(M_{\xi_0}; \mathfrak{t})$ the Chern class of the principal \(\mathbb{T}^k\)-bundle $\mu^{-1}(\xi_0)\to M_{\xi_0}$. Equivalently, after choosing a basis of \(\mathfrak{t}\), we may write
\[
    \Omega=(\Omega_1,\ldots,\Omega_k),
    \qquad
    \Omega_j\in \mathrm{H}^2(M_{\xi_0};\mathbb{R}),
\]
and for \(\eta\in\mathfrak{t}^*\) we use the notation
\[
    \langle \eta,\Omega\rangle
    =
    \sum_{j=1}^k \eta_j\Omega_j.
\]

Finally, Duistermaat and Heckman also explain how to relate their computations to Hamiltonian actions of compact, connected Lie groups $G$. Let us fix a maximal torus $\mathbb{T} \subset G$, and denote by
$
i \colon \mathfrak{t} \hookrightarrow \mathfrak{g}
$
the corresponding inclusion of Lie algebras. By fixing an $\operatorname{Ad}$-invariant, bilinear, non-degenerate form $\kappa$ on $\mathfrak{g}$, we obtain isomorphisms $\kappa^\flat \colon \mathfrak{g} \to \mathfrak{g}^*, \kappa^\flat \colon \mathfrak{t} \to \mathfrak{t}^*$, with inverses denoted by $\kappa^\sharp$. These maps may be used to construct a section $p^* \colon \mathfrak{t}^* \to \mathfrak{g}^*$ defined by $p^* = \kappa^\flat \circ i \circ \kappa^\sharp$. Recall that, over the set of regular points $\mathfrak{g}_{\mathrm{reg}}$, the set $\mathfrak{t}_{\mathrm{reg}} = \mathfrak{g}_{\mathrm{reg}} \cap \mathfrak{t}$ is a slice for the adjoint action. Since $\kappa$ intertwines the adjoint and coadjoint actions, the corresponding set $\mathfrak{t}_{\mathrm{reg}}^* \coloneqq \kappa^\flat(\mathfrak{t}_{\mathrm{reg}})$ is a slice for the coadjoint action on $\mathfrak{g}_{\mathrm{reg}}^* \coloneqq \kappa^\flat(\mathfrak{g}_{\mathrm{reg}})$.

Now, the centralizer of a regular value $\xi \in \mathfrak{g}^*$ of $\mu$ is a maximal torus; in particular, $\xi \in \mathfrak{g}_{\mathrm{reg}}^*$. Therefore, there exists $g \in G$ such that $\alpha = \operatorname{Ad}_g^* \xi \in \mathfrak{t}_{\mathrm{reg}}^*$. By the equivariance of the moment map,
\[
\mu^{-1}(\xi)
=
\mu^{-1}(\operatorname{Ad}_{g^{-1}}^* \alpha)
=
\rho_g\bigl(\mu^{-1}(\alpha)\bigr).
\]
Since the corresponding stabilizers are also conjugate, we conclude that $\rho_g$ induces a symplectomorphism $\rho_g \colon M_\alpha \to M_\xi$. Thus, the variation of the cohomology class on $M_\xi$ can be read from the corresponding variation on $M_\alpha$.

With these preliminaries in place, we may now state and prove the cosymplectic analogue of Duistermaat--Heckman's variation in cohomology. We note that the classical computation of Duistermaat and Heckman only concerns the variation in cohomology of the symplectic structure. In the cosymplectic setting, however, we do not only have a closed, leafwise symplectic form but also the defining one-form $\alpha$: therefore, the variation in cohomology of the cosymplectic structure also encodes the transverse component to the symplectic foliation.

\begin{theorem} \label{thm:var-cosymp-struct}
    Let $(M, \alpha, \beta)$ be a cosymplectic manifold equipped with a Hamiltonian action of a torus $\mathbb{T}^k$. Let $\xi, \xi_0 \in \mathfrak{t}^*$ lie in the same connected component of the set of regular values of $\mu$. Then, we have
    \begin{equation} \label{eq:var-cosymp-struct}
        [\alpha_\xi] = [\alpha_{\xi_0}], \qquad  [\beta_\xi] = [\beta_{\xi_0}] + \langle \xi - \xi_0, \Omega \rangle,
    \end{equation}
    where $\Omega$ is the Chern class of the fibration $\mu^{-1}(\xi_0) \to M_{\xi_0}$.
\end{theorem}

\begin{proof}
    Let us consider the symplectic thickening $M \times \mathbb{S}^1$ with symplectic form $\omega = \beta + \diff \theta \wedge \alpha$. The classical Duistermaat--Heckman theorem implies that the variation of the symplectic form is given by
    \begin{equation} \label{eq:sympl-dh}
        \omega_\xi = \omega_{\xi_0} + \langle \xi - \xi_0, \widehat{\Omega} \rangle,
    \end{equation}
    with $\widehat{\Omega}$ the first Chern class of the fibration $\hat{\mu}(\xi_0) \to \widehat{M}_{\xi_0}$.

    We can now specialize this result to the specific case of the symplectic thickening $\widehat{M}$ of a cosymplectic manifold. From diagram \eqref{eq:triv-red-thick} we observe that the fibration $\hat{\mu}^{-1}(\xi_0)$ can be described as the pullback $\operatorname{pr}_1^*(\mu^{-1}(\xi_0))$. In particular, the Chern classes are related by $\widehat{\Omega} = \operatorname{pr}_1^* \Omega$ with $\Omega$ the Chern class of the fibration $\mu^{-1}(\xi_0) \to M_{\xi_0}$.

    Now, we may recover the cohomology classes of the reduced cosymplectic structures in $M_\xi$ by virtue of \Cref{rmk:cosymp-struct-from-symp-thick}. Observe that we have
    \begin{equation*}
        \operatorname{pr}_1^* [\alpha_\xi] - \operatorname{pr}_1^* [\alpha_{\xi_0}] = \iota_{\partial_\theta} [\omega_\xi - \omega_{\xi_0}] = \iota_{\partial_\theta} \langle \xi - \xi_0, \widehat{\Omega} \rangle = \langle \xi - \xi_0, \iota_{\partial_{\theta}} \operatorname{pr}_1^* \Omega \rangle = 0
    \end{equation*}
    since $\iota_{\partial_\theta} \operatorname{pr}_1^* \Omega = 0$. As the map $\operatorname{pr}_1^* \colon \mathrm{H}^\bullet(M) \to \mathrm{H}^\bullet(M \times \mathbb{S}^1)$ is injective, we conclude the first part of the result.
    
    From this equation we may explicitly derive the variation of the cohomology class $[\beta_\xi]$. From expression \eqref{eq:sympl-dh} and the previous result,
    \begin{align*}
        \operatorname{pr}_1^* [\beta_\xi] - \operatorname{pr}_1^* [\beta_{\xi_0}] &= [\omega_\xi] - [\omega_{\xi_0}] - \diff \theta \wedge \operatorname{pr}_1^* [\alpha_{\xi}] - \diff \theta \wedge \operatorname{pr}_1^* [\alpha_{\xi_0}] \\
        &= \langle \xi - \xi_0, \operatorname{pr}_1^* \Omega \rangle \\
        &= \operatorname{pr}_1^* \langle \xi - \xi_0, \Omega \rangle.
    \end{align*}
    Once again, since $\operatorname{pr}_1^*$ is injective the result follows.
\end{proof}

\begin{remark}
    Since our proof only relies on the classical result of Duistermaat--Heckman and the naturality of the fibration $\operatorname{pr}_1^*\mu = \hat{\mu}$ following \Cref{thm:reduction-commutes-thickening}, we obtain a similar description to Duistermaat and Heckman \cite[p.\,264]{duistermaat-variationcohomology-1982} for the variation in cohomology for Hamiltonian actions of compact, connected Lie groups $G$.
\end{remark}

Recall that the form $\alpha$, or rather its cohomology class $[\alpha]$, is closely related to the modular period described in \Cref{thm:mapping-torus}. As a consequence of \Cref{thm:var-cosymp-struct} we have the following result.

\begin{corollary}
    Let $(M, \alpha, \beta)$ be a cosymplectic manifold described as a mapping torus as in \Cref{thm:mapping-torus} and under the assumptions of \Cref{thm:var-cosymp-struct}. Then, the modular period of the reduced spaces $M_\xi$ is constant, i.e. $\operatorname{mp}(\alpha_{\xi}, \beta_\xi) = \operatorname{mp}(\alpha_{\xi_0}, \beta_{\xi_0})$ for all $\xi \in U_{\xi_0}$.
\end{corollary}

\begin{proof}
    Recall that the various reduced spaces are identified by means of a $\mathbb{T}^k$-invariant connection. The modular period $\operatorname{mp}(\alpha, \beta)$ of a cosymplectic manifold may be computed as the integral
    \begin{equation*}
        \operatorname{mp}(\alpha, \beta) = \int_{\gamma} \alpha
    \end{equation*}
    for any $\gamma \colon \mathbb{S}^1 \to M$ such that\footnote{
       Equivalently, $\gamma$ is any loop whose projection to the base circle of the mapping torus has degree one. This formulation removes the need to keep track of base points in the fundamental groups.
    } $[\pi \gamma] = 1 \in \mathbb{Z} \cong \pi_1(\mathbb{S}^1)$. But then, this equation together with expression \eqref{eq:var-cosymp-struct} directly implies the result.
\end{proof}

The variation of the cohomology class of the reduced spaces was originally used by Duistermaat and Heckman to derive a formula for the variation of the symplectic volume and to relate it to the pushforward of the symplectic measure. We conclude this section by proving an analogous result in the cosymplectic setting.

\begin{theorem}
    Let $(M, \alpha, \beta)$ be a cosymplectic manifold in the assumptions of \Cref{thm:var-cosymp-struct}. If we denote by $\lambda_{(M, \alpha, \beta)}$ the standard cosymplectic measure and by $\lambda_{\mathrm{Leb}}$ the Lebesgue measure in $\mathfrak{t}^*$, the Radon--Nykodim derivative of the pushforward measure $\mu_* \lambda_{(M, \alpha, \beta)}$ is given by
    \begin{equation}
        f(\xi) = \frac{\diff \mu_* \lambda_{(M, \alpha, \beta)}}{\diff \lambda_{\mathrm{Leb}}}(\xi) = \int_{\mu^{-1}(\xi)} \alpha_\xi \wedge \frac{\beta_\xi^{n - l}}{(n - l)!}.
    \end{equation}
    In particular, \Cref{thm:var-cosymp-struct} implies the function $f$ is polynomial.
\end{theorem}

\begin{proof}
    Let us denote by $(\widehat{M}, \widehat{\omega})$ the symplectic thickening of $M$ following \Cref{prop:equivariant-symplectic-thickening} with induced moment map $\hat{\mu} = \operatorname{pr}_1^* \mu$. Let us begin by showing the relation between the Radon--Nykodim derivative of $\mu_* \smash{\lambda_{(\widehat{M}, \widehat{\omega})}}$ and $\mu_* \lambda_{(M, \alpha, \beta)}$. Direct computation taking equation \eqref{eq:thicken-sympl-vol} in mind shows
    \begin{align*}
        \hat{\mu}_* \lambda_{(\widehat{M}, \widehat{\omega})}(V) = \int_{\hat{\mu}^{-1}(V)} \frac{\widehat{\omega}^{n + 1}}{(n + 1) !} = \int_{\mu^{-1}(V) \times \mathbb{S}^1} \frac{(n + 1) \diff \theta \wedge \alpha \wedge \beta^n}{(n + 1)!} &= -\int_{\mu^{-1}(V)} \alpha \wedge \frac{\beta^n}{n!} \\
        &= -\mu_* \lambda_{(M, \alpha, \beta)}(V).
    \end{align*}
    Therefore, if $\hat{f}$ denotes the Radon--Nykodim derivative $\smash{\mu_* \lambda_{(\widehat{M}, \widehat{\omega})}}$ with respect to $\lambda_{\mathrm{Leb}}$, the previous discussion shows that $f = - \hat{f}$.
    
    Now, from the classical Duistermaat--Heckman theorem \cite[Prop.\,3.2]{duistermaat-variationcohomology-1982} we have that the Radon--Nykodim derivative satisfies the formula
    \begin{equation}
        \hat{f}(\xi) = \int_{\hat{\mu}^{-1}(\xi)} \frac{\widehat{\omega}_\xi^{n + 1 - k}}{(n + 1 - k)!}
    \end{equation}
    This fact together with \Cref{thm:reduction-commutes-thickening} yields
    \begin{align*}
        \hat{f}(\xi) &= \int_{\hat{\mu}^{-1}(\xi)} \frac{\widehat{\omega}_\xi^{n + 1- k}}{(n + 1 - k)!} \\
        &= \int_{\mu^{-1}(\xi) \times \mathbb{S}^1} \frac{(n + 1 - k) \diff \theta \wedge \alpha_\xi \wedge \beta_\xi^{n - k}}{(n + 1 - k)!} = - \int_{\mu^{-1}(\xi)} \alpha_\xi \wedge \frac{\beta_\xi^{n - k}}{(n - k)!}.
    \end{align*}
    This concludes the proof.
\end{proof}

\begin{remark}
    The previous result can be further refined if $M$ satisfies the mapping-torus assumptions of \Cref{thm:mapping-torus}. If we denote by $(L, \omega_L)$ a compact symplectic leaf, \Cref{lem:modular-weight-integral} implies that we have the relation
    \[
        \mu_* \lambda_{(M,\alpha,\beta)}
        =
        \operatorname{mp}(\alpha,\beta)\,
        \mu_* \lambda_{(L,\omega_L)}.
    \]
    Thus, in the mapping-torus case, the cosymplectic Duistermaat--Heckman measure is the ordinary symplectic Duistermaat--Heckman measure of a compact leaf, multiplied by the modular period.
\end{remark}

\begin{example}
    Let
    \[
        M=\mathbb{S}^1\times\mathbb{C}\mathrm{P}^n,
        \qquad
        \alpha=\diff t,
        \qquad
        \beta=\omega_{\mathrm{FS}},
    \]
    where $\omega_{\mathrm{FS}}$ denotes the Fubini--Study form pulled back from $\mathbb{C}\mathrm{P}^n$. Then $(M,\alpha,\beta)$ is a compact cosymplectic manifold. The standard Hamiltonian action of $\mathbb{T}^n$ on $\mathbb{C}\mathrm{P}^n$ extends trivially to $M$, and its moment map is $\mu(t,x)=\mu_{\mathrm{FS}}(x)$. Hence $\mu(M)=\mu_{\mathrm{FS}}(\mathbb{C}\mathrm{P}^n)$,
    which is the standard simplex in $\mathbb{R}^n$.

    Moreover,
    \[
        \mu_*\left(\alpha\wedge\frac{\omega_{\mathrm{FS}}^n}{n!}\right)
        =
        \left(\int_{\mathbb{S}^1}\alpha\right)
        (\mu_{\mathrm{FS}})_*
        \left(\frac{\omega_{\mathrm{FS}}^n}{n!}\right).
    \]
    Thus, the cosymplectic Duistermaat--Heckman measure is the usual Duistermaat--Heckman measure of $\mathbb{C}\mathrm{P}^n$, multiplied by the length of the Reeb circle.
\end{example}

\bibliographystyle{alpha}
\bibliography{equivariant-cosymplectic-manifolds}

\end{document}